\documentclass[12pt, reqno]{amsart}

\usepackage{amsmath,amsthm,amssymb,amsfonts,verbatim}

\usepackage[colorlinks,
            linkcolor=blue,
            anchorcolor=blue,
            citecolor=blue
            ]{hyperref}
\usepackage[hmargin=1.2in,vmargin=1.2in]{geometry}
\usepackage[shortlabels]{enumitem}

\setlist[enumerate, 1]{label = (\arabic*), font = \normalfont}

\def \b {\beta}

\def \F {{\mathbb F}}
\def \Z {{\mathbb Z}}

\def \supp {{\rm supp}}
\DeclareMathOperator{\Diff}{Diff}

\DeclareMathOperator{\Div}{Div}
\DeclareMathOperator{\End}{End}
\DeclareMathOperator{\Ker}{Ker}
\DeclareMathOperator{\Aut}{Aut}
\DeclareMathOperator{\Gal}{Gal}

\DeclareMathOperator{\Fr}{Fr}
\DeclareMathOperator{\id}{id}

\newtheorem{theorem}{Theorem}[section]
\newtheorem{corollary}[theorem]{Corollary}
\newtheorem{proposition}[theorem]{Proposition}
\newtheorem{lemma}[theorem]{Lemma}

\newtheorem{remark}[theorem]{Remark}

\providecommand{\mcal}{\mathcal}

\providecommand{\mrm}{\mathrm}

\providecommand{\mbb}{\mathbb}
\providecommand{\mbf}{\mathbf}
\providecommand{\bm}{\boldsymbol}

\providecommand{\Z}{{\mathbb Z}}

\providecommand{\F}{{\mathbb F}}
\providecommand{\N}{{\mathbb N}}

\providecommand{\bbp}{{\mathbb P}}

\providecommand{\trans}{\ensuremath{^{\mrm T}}}

\newcommand{\gauss}[1]{\left\lfloor #1 \right\rfloor}

\begin{document}

\title[Automorphism group of the Abd\'on--Torres curve]{The Automorphism groups of a maximal function field with the second largest genus and its algebraic geometry codes}
\thanks{}
\author{Xubin Hu} \address{School of Mathematical Sciences, University of Science and Technology of China, Hefei China
230026}\email{xbh5168@mail.ustc.edu.cn}
\author{Liming Ma}\address{School of Mathematical Sciences, University of Science and Technology of China, Hefei China
230026}\email{lmma20@ustc.edu.cn}

\maketitle

\begin{abstract}
In this manuscript, we investigate the automorphism group of a maximal function field with the second largest possible genus over finite field of even characteristic, which is called the Abd\'on--Torres function field. As an application, we determine the automorphism groups of one-point algebraic geometry codes from such a maximal function field. It turns out that the automorphism groups of one-point algebraic geometry codes agree with that of the Abd\'on--Torres  function field except for the trivial cases. Moreover, we provide a family of maximal function fields with explicit defining equations via considering fixed subfields with respect to some subgroups of automorphism group of the Abd\'on--Torres function field.
\end{abstract}

\section{Introduction}
Let $\F_\ell$ be the finite field with $\ell$ elements and $ \overline \F_\ell$ be its algebraic closure. Let  $F/ \F_\ell$ be an algebraic function field of one variable with genus $g$ over the full constant field $\F_\ell$. The celebrated Hasse--Weil Theorem states that the number $N(F)$ of rational places of $F$ satisfies the following inequality
\[|N(F) - (\ell  + 1) | \leq 2 g \sqrt \ell.\]
Any function field $F/ \F_\ell$  with the number of rational places $N(F)$ attaining the upper Hasse--Weil bound $\ell +1+ 2 g \sqrt \ell$ is called \emph{maximal}.
If $F$ is maximal, then either $g = 0$ or $\ell=q^2$ is a square of a prime power $q$.

Maximal function fields are widely studied not only in the viewpoints of number theory and algebraic geometry, but also due to its various applications in coding theory and cryptography \cite{Go1983,LWX17,TVN2007, TVN2019}.
In order to construct long algebraic geometry codes with good properties, one might ask what are the possible values of $g$ for maximal function fields over $\F_{q^2}$.
It was shown that $g \leq  g_1 = (q-1)q/2$ by Ihara in \cite{Ih1981}.
The well-known Hermitian function field $H= \F_{q^2}(u,v)$ defined by $$v^q + v=u^{q+1}$$ is a maximal function field with the largest possible genus $(q-1)q/2$. Moreover, it is the unique maximal function field $\F_{q^2}$ with genus $ (q-1)q/2$ up to isomorphism \cite{RS94}.

It was conjectured in \cite{XS95} and proven in \cite{FT1996} that if $F$ is a maximal function field over $\F_{q^2}$ with genus $g$, then $g = q(q-1)/2$ or $g \leq g_2 =\lfloor(q-1)^2/4\rfloor$, here $\lfloor x\rfloor$ is the integer part of $x\in \mathbb{R}$.
For an odd prime power $q$, it was proved in \cite{FGT1997} the maximal function field over $\F_{q^2}$ with genus $g = (q-1)^2 / 4$ is isomorphic to $X_2= \F_{q^2}(x,y)$ defined by
\[y^q + y=x^{(q+1)/2}.\]
For an even prime power $q=2^n$, Abd\'on and Torres proved in \cite{AT1999} that a maximal function field over $\F_{q^2}$ with genus $ g  =  q(q-2)/4$ is $\F_{q^2}$-isomorphic to $Y_2 = \F_{q^2}(x,y)$ defined by
\[y^{q/2} + y^{q/2^2} + \dots + y^2 + y=x^{q+1},\]
provided that $q/2$ is a Weierstrass pole number of some rational place of $Y_2 \overline {\F}_{q^2}$. In this paper, we call such a function field as the Abd\'on-Torres function field.
 In \cite{KT2002}, Korchm\'aros and Torres proved that the genus $g$ of a maximal function field over $\F_{q^2}$ is either $g_1, g_2$, or $\leq g_3 = \lfloor(q^2 - q+4)/6\rfloor$, and the following three function fields attain the bound $g_3$:
 \begin{align*}
	X_3\,&{:=}\, \F_{q^2} (x,y), \quad x^{(q+1)/3} + x^{2 (q+1)/3} + y^{1+q} = 0,\quad &\text{ if } q \equiv 2 \pmod 3, \\
	Y_3\,&{:=}\, \F_{q^2}(x,y), \quad y^q - yx^{2(q-1)/3} + x^{(q-1)/3} = 0, \quad &\text{ if } q \equiv 1 \pmod 3, \\
Z_3\,&{:=}\, \F_{q^2}(x,y), \quad y^q + y + (\sum_{j =  1}^t x^{q  /3^j})^2=0,\quad &\text{ if } q \equiv 0 \pmod 3.
\end{align*}
All these maximal function fields are Galois subfields of the Hermitian function fields.

A well-known result of Serre implies that any function subfield of a maximal function field is maximal as well \cite[Proposition 6]{La1987}.
The automorphism group of the Hermitian function field $H= \F_{q^2}(u,v)$ is isomorphic to the projective unitary group $\operatorname{PGU}_3(\F_{q^2})$ with large order $q^3(q^2-1)(q^3+1)>16g_1^4$ and many subgroups \cite{Sti73}.
Hence, many maximal function fields can be constructed systematically by considering fixed subfields of subgroups of automorphism groups of the Hermitian function fields \cite{CKT1999,CKT2000,GSX2000,AQ2004,BMXY2013,DMZ2019,MX2019,MZ2020}.
In the early literatures, maximal function fields are proved to be subfields of the Hermitian function fields or undetermined.
However, Giulietti and Korchm\'aros provide an explicit construction of maximal function field with genus $g=\frac{1}{2}(n^3+1)(n^2-2)+1$ over $\F_{q^2}$ which is not a subfield of the Hermitian function field for $q=n^3>8$ in \cite{GK2009}.
Its automorphism group is determined explicitly as well in \cite{GK2009} and has a large size $n^3(n^3+1)(n^2-1)(n^2-n+1)$, which is roughly $4g^2$. Maximal subfields of Giulietti--Korchm\'aros function field have been determined in \cite {GQZ2016}.
Moreover, the Giulietti--Korchm\'aros function field has been generalized by Garcia, G\"uneri and Stichtenoth in \cite{GGS2010} and by Beelen and Montanucci in \cite{BM2018}.

Recently, the automorphism groups of the maximal function fields $X_3,Y_3$ and $Z_3$ with the third largest genus have been determined in \cite{BMV2023,BMV2024,BMV2025}.
It is very strange that the automorphism groups of maximal function fields $X_2$ and $Y_2$ with the second largest genus were not mentioned in the above literatures.
Finally, we found that the automorphism group of $X_2$ is isomorphic to a semiproduct of a cyclic normal subgroup $C_{(q+1)/2}$ of order $(q+1)/2$ and the projective linear group $\text{PGL}_2(\F_{q^2})$ in \cite[Theorem 12.11]{CurveBook}. However, it seems that the automorphism group $\Aut (Y_2 / \F_{q^2})$ of the Abd\'on--Torres function field $Y_2$ is not well-documented in the past researches.
Therefore, we determine the automorphism group of the maximal Abd\'on--Torres function field $Y_2$ over $\F_{q^2}$ and provide a family of maximal function fields with explicit defining equations via considering its Galois subfields.

By choosing divisors carefully, an automorphism of the function field can induce an automorphism of the algebraic geometry code from such a function field.
In 1990, Stichtenoth determined the automorphism group of rational algebraic geometry codes in \cite{Sti90}.
In 1995, Xing determined the automorphism group of the one-point algebraic geometry codes from elliptic curves in \cite{Xing95a} and the one-point Hermitian codes in \cite{Xing95b}. In 1998, Wesemeyer developed a systematic method to compute the automorphism group of a large family of one-point algebraic geometry codes from plane curves and generalized Xing's result for some cases in \cite{Wes98}.  In 2006, this result has been further generalized by Joyner and Ksir in \cite{JK06}.
In fact, they discovered that one can lift an automorphism of algebraic geometry codes to an automorphism of the associated curve when the divisor is very ample.
In 2008, Giulietti and Korchm\'aros systematically studied the lifting of semi-linear automorphisms of algebraic geometry codes that preserve Hamming weights given in \cite{PHB98}. They proved that such automorphisms of Deligne--Lusztig codes can be lifted in most cases \cite{GK08}.
Aside from these framework, several cases were studied as well, including the algebraic geometry codes related to the curve $y^q + y = x^{q^{r}+1}$ in \cite{KKO01}; the automorphism group of generalized algebraic geometry codes related to hyperelliptic curves in \cite{PS06}; the construction of Hermitian codes with automorphism group isomorphic to $\text{PGL}_2(\F_{q^2})$ for odd $q$ in \cite{KS2017}, etc.
In this paper, we shall determine the automorphism group of the one-point algebraic geometry codes from the Abd\'on--Torres function field $Y_2$.

This paper is organized as follows.
In Section \ref{preliminary}, we provide preliminary results on the theory of function fields, algebraic geometry codes and their automorphism groups, and the maximal curve with the second largest genus given in \cite{AT1999}.
In Section \ref{sec:3}, we compute explicitly the Weierstrass semigroups at all rational places and determine the automorphism group of the maximal curve with the second largest genus. In Section \ref{sec:4}, we investigate fixed subfields of the maximal function field of the second largest genus and we provide explicit equation of some of these Galois subfields. In the final section, as a major application, we completely determine the automorphism group of the one-point algebraic geometry codes from the Abd\'on--Torres function field $Y_2$.

\section{Preliminaries}\label{preliminary}
In this section, we collect some preliminary results on global function field and ramification theory, algebraic geometry codes and their automorphism groups, and a maximal curve with the second largest genus. We prefer the language of function fields. For more details, please refer to \cite{Sti2009,NX01}.

\subsection{Global function fields and ramification theory}
Let $\ell$ be a prime power, $\F_{\ell}$ be the finite field with $\ell$ elements and
$\overline \F_\ell$ be its algebraic closure. 
A function field $F$ of one variable over $\F_\ell$ is a finite algebraic extension of the rational function field $\F_{\ell}(x)$ over $\F_{\ell}$.
If any element in $F\setminus \F_\ell$ is transcendental over $\F_\ell$, then it is called a global function field and denoted by $F/\F_{\ell}$.

Let $\mathcal O$ be a valuation ring of $F/\F_{\ell}$. It has a unique maximal ideal $P$, which is called a place of $F$.
Such a valuation ring $\mathcal O$ can be denoted by $\mathcal O_P$ and its corresponding normalized discrete valuation is denoted by $\nu_P$. The degree $\deg(P)$ of $P$ is defined to be the extension degree of residue class field $\kappa_P=\mathcal{O}_P/P$ over $\F_{\ell}$. If $\deg(P)=[\kappa_P:\F_{\ell}]=1$, then $P$ is called a rational place of $F$.
The set of places of $F$ is denoted by $\bbp _F$ and rational places form a subset $\bbp _F^{(1)} $ of $\bbp _F$.
 A finite sum of places $D = \sum_{P \in \mbb P_F} n_P P$ with $n_P\in \mathbb{Z}$ is called a  divisor of $F$, and all divisors form a free abelian group $\Div(F)$. The set $\supp(D)=\{P\in \mbb P_F: n_P\neq 0\}$ is called the support of $D$. 
 The degree of $D$ is defined by $\deg D=\sum_{P \in \mbb P_F} n_P \deg(P)$. The subgroup of  divisors with degree $0$ is denoted by $\Div ^0 (F)$.

For an element $f \in F^* = F \setminus \{0\}$, a place $P$ is a zero of $f$ if and only if $\nu_P(f) > 0$. Denote the set of all zeros of $f$ by $Z(f)$. Then the divisor $(f)_0^F = \sum _{P \in Z(f)} \nu_P(f) P$ is called the zero divisor of $f$ and the divisor $(f)_\infty^F = (1/f)_0^F$ is called the pole divisor of $f$.
 The principal divisor of $f$ is $(f)^F = (f)_0^F - (1/f)_0^F$, and its degree is $0$. If there is no ambiguity, then we can drop the superscript and simply write $(f)$, $(f)_0$ and $(f)_\infty$. All principal divisors form a subgroup $\operatorname{Princ} (F)$ of $\Div ^0 (F)$, and denote the divisor class group of degree $0$ of $F$ by $\operatorname{Pic}^0 (F) = \Div^0 (F) / \operatorname{Princ} (F)$.

The Riemann--Roch space associated with a divisor $D$ of $F / \F_\ell$ is defined by
\[\mcal L(D)= \{ f \in F^*\colon (f) + D \geq 0\} \cup \{0\}.
\]
It is a finite-dimensional vector space over $\F_\ell$ with dimension $\dim \mathcal{ L}(D) $. By the Riemann--Roch Theorem,
$\deg D + 1 - \dim \mathcal{ L}(D)$
has a non-negative least upper bound for all divisors $D$.
It is the most important birational invariant of function field $F$, which is called the genus of $F$ and denoted by $g(F)$.
If $ \deg D \geq 2 g(F)  -1$, then \[\dim \mathcal{ L}(D) = \deg D + 1 -  g(F).\]

The automorphism group over $\F_\ell$ of a global function field $F/\F_\ell$ is defined by
\[\Aut (F / \F_\ell) = \{ \sigma:    \sigma \text{ is an } \F_\ell\text{-automorphism of } F \}.
\]
For each $\sigma \in \Aut (F/ \ell)$, the set $\sigma (P) = \{ \sigma (z) \colon z \in P\}$ is a place of $F$ with degree $\deg(P)$ as well.
This action naturally extends to $\Div(F)$ by $\Z$-linearity.
From the fact $\nu_{\sigma(P)}(\sigma(f))=\nu_{P}(f)$ given in \cite[Lemma 3.5.2]{Sti2009}, we have $\sigma (\mcal L(D)) = \mcal L(\sigma (D))$ for any $\sigma \in \Aut (F/ \F_\ell)$ and $D \in \Div(F)$. For a subgroup $G$ of $\Aut (F/ \F_\ell)$, the fixed subfield of $F$ under $G$ is defined by
\[
	F^G  = \{ z \in F \colon  \phi (z) = z \text{ for all } \phi \in G\}.
\]
 By Galois theory, $F / F^G$ is a finite Galois extension.

Let $E/\F_\ell $ be a finite Galois extension of $F/\F_\ell$ with Galois group $\Gal (E/F)$. Let $Q$ be a place of $E$ and $P$ be a place of $F$ lying under $Q$. For a prime element $t_P$ of $P$, the integer $e (Q | P) = \nu _Q(t_P)$ is the ramification index of $Q$ over $P$.
The relative degree $f(Q | P)$ of $Q$ over $P$ is equals to $[\kappa_Q: \kappa_P]$. 
The Hurwitz genus formula states that
\[2 g(E) - 2 = (2 g(F) - 2) [E:F] + \deg \Diff (E/F),\]
where $\Diff (E /F) = \sum_{P \in \bbp_F} \sum_{Q|P}  d(Q | P) Q$
is the different divisor of $E /F$. The different exponent $d(Q|P)$ can be calculated via higher ramification groups by Hilbert ramification theory.
The group
\[G_{-1} (Q | P) = \{\sigma \in \Gal (E/F) \colon \sigma (Q) =Q\}\]
is called the decomposition group of $Q|P$, and the subgroup
\[
G_j (Q | P) = \{\sigma \in \Gal (E / F) \colon  \nu_Q (\sigma (z) - z) \geq j+1  \text{ for all } z \in \mathcal{O}_Q\}
\]
is the $j$-th higher ramification group for each $j \in \N$. They form a decreasing sequence $G_{-1}(Q | P) \supseteq G_0(Q | P) \supseteq \cdots \supseteq G_m(Q | P) = G_{m+1}(Q | P) = \cdots = \{\id\}$ with $m\geq -1$. By Hilbert's Different Theorem \cite[Theorem 3.8.7]{Sti2009}, the different exponent can be computed as
\[d (Q |P) = \sum_{j =0}^\infty (|G_j| - 1).\]

Moreover, if $E/F$ is a finite abelian extension, let $a(Q|P)$ be the least non-negative integer $a$ such that ramification groups $G_j$ become trivial for all $j\geq a$, then the integer
\[
	c_P(E/F) := \frac {d(Q|P) + a(Q|P)} {e(Q|P)}
\]
is called the conductor exponent of $P$ in $E/F$. The conductor of $E / F$ is defined by
$$\text{Cond} (E/F) := \sum _{P \in \bbp_F} c_P (E/F) P.$$

\subsection{Algebraic geometric codes and their automorphism groups}

Let $F / \F_{\ell} $ be a global function field. Let $\mcal D = \{P_1, P_2, \dots, P_n\}$ be an ordered set or a list of $n$ distinct rational places of $F$.
For a divisor $G$ with $P_j \notin \supp(G)$ for each $1\le j\le n$, the algebraic geometry code associated with $\mcal D$ and $G$ is denoted by
\[C(\mcal D , G) := \{ (f (P_1), f(P_2), \dots,f(P_n)) \colon f \in \mcal L(G)\} \subseteq \F_{\ell} ^n.\]
Consider the action of symmetric group $S_n$ of $n$ elements on the vector space $\F_\ell^n$ as
\[
 \tau (\bm c):= (c _{\tau (1)}, c_{\tau (2)}, \dots, c_{\tau (n)} ),
\]
for all $\tau \in S_n$ and $\bm c = (c_1,c_2, \dots, c_n) \in \F_\ell^n$.
The \emph{automorphism group} $\operatorname{Aut}(C)$ of code $C$ is defined by
\[
\operatorname{Aut} (C) := \{\tau  \in S_n \colon \tau (C) = C\}.
\]
Let $D = \sum_{P \in\mcal D} P$ and let $\Aut _{D,G} (F/ \F_\ell)$ be a subgroup of $\Aut (F / \F_\ell)$ defined by
\[
	\Aut _{D,G} (F/ \F_\ell) := \{ \sigma \in \Aut (F/ \F_\ell)\colon \sigma (D) = D, \sigma (G) = G \}.
\]
By \cite[Proposition 8.2.3]{Sti2009}), this subgroup acts on the code $C = C (\mcal D, G)$  by
\[\sigma(f(P_1), f(P_2), \dots, f(P_n)):= (f (\sigma (P_1)), f (\sigma (P_2)), \dots, f (\sigma (P_n))  ),
\]
for any $\sigma \in \Aut (F/ \F_\ell)$ and $f\in \mathcal{L}(G)$.
Hence, this yields a homomorphism from $\Aut _{D, G} (F / \F_\ell)$ to $\Aut (C)$.
If $n > 2 g(F) - 2$, then it is an injective homomorphism. Therefore,  $\Aut_{D,G} (F/ \F_\ell)$ can be viewed as a subgroup of $\Aut (C(\mcal D , G))$.

\subsection{A maximal curve with the second largest genus: the Abd\'on--Torres curve}
Let $q$ be a prime power of $2$.
Abd\'on and Torres showed that the maximal curve with the second largest genus over the finite field $\F_{q^2}$ with even characteristic is $\F_{q^2}$-isomorphic to the nonsingular plane curve $$y^{q/2} + y^{q/4} + \dots + y=x^{q+1},$$ provided that $q/2$ is a Weierstrass non-gap at some point of the curve \cite{AT1999}.
For simplicity, we call this curve (or function field) as the Abd\'on--Torres curve (or function field).

\begin{proposition}\label{prop:ram_Y2}
	Let $q \geq 4$ be a power of $2$. Let $h(y) = y^{q/2} + y^{q/2^2} + \dots + y^2 + y$ be a polynomial in $\F_{q^2}[y]$.
	Let $Y_2=\F_{q^2}(x,y)$ be the Artin--Schreier function field over $\F_{q^2}$ defined by the equation
	\[h(y)=x^{q+1}.\]
	Then the following results hold true.
	
	\begin{itemize}
		\item[(1)]   The infinite place $\infty$ of $\F_{q^2}(x)$ is totally ramified place in the extension $Y_2/\F_{q^2}(x)$. Let $P_\infty$ be the place of $Y_2$ lying over $\infty$.
Then the ramification index of $P_\infty|\infty$ is $e(P_\infty|\infty)=q/2$ and the different exponent of $P_\infty|\infty$ is $d(P_\infty|\infty)=(q^2-4)/2$.
		\item[(2)]   All other places of $\F_{q^2}(x)$ are unramified in $Y_2$. 
	\item[(3)]   The genus of $Y_2$ is $g({Y_2}) = q(q-2) / 4.$
\item[(4)]  $Y_2$ is a fixed subfield of the Hermitian function field, and hence a maximal function field over $\F_{q^2}$.
		\item[(5)]   There are exactly $1+q^3/2$ rational places of $Y_2$. Namely, the common pole $P_\infty$ of $x$ and $y$ is a rational place; for any $a\in \F_{q^2}$, there are exactly $q/2$ elements $b$ in $\F_{q^2}$ such that $h(b)=a^{q+1}$, i.e., the common zero $P_{a,b}$ of $x-a$ and $y-b$ with $h(b)=a^{q+1}$ is a rational place of $Y_2$.
	
	\end{itemize}
\end{proposition}


\begin{proof}
	\begin{enumerate}
		\item Let $\nu _\infty$ be the discrete valuation with respect to $\infty$. It is clear that $\nu _\infty (x^{q+1}) = -(q+1)$, which is not divisible by $2$. By \cite[Proposition 3.7.10]{Sti2009}, the rational place $\infty$ is totally ramified in $Y_2 / \F_{q^2}(x)$ with ramification index $e (P_\infty \mid \infty) =[Y_2:\F_{q^2}(x)]= q/2 $ and  different exponent
		\[d(P_\infty \mid \infty) = (q+2)\left(\frac{q}{2} - 1\right) = \frac {q^2 - 4}2.\]
\item  For any finite place $R$ of $\F_{q^2} (x)$, we have $\nu_R(x^{q+1}) \geq 0$. By \cite[Proposition 3.7.10]{Sti2009}, all other places $R$ of $\F_{q^2} (x)$ are unramified in $Y_2/\F_{q^2}(x)$.

\item The genus of $Y_2$ follows from \cite{AT1999} or Hurwitz genus formula as follows
$$2g(Y_2)-2=\frac{q}{2} [2g(\F_{q^2}(x))-2]+\deg {\rm Diff}(Y_2/\F_{q^2}(x))=\frac{q}{2} (0-2)+\frac{q^2-4}{2}. $$

\item    Let $x=u$ and $y=v^2+v$. It is clear that $u^{q+1}=x^{q+1}=y^{q/2} + y^{q/4} + \dots + y=v^q+v$.
		The function field $H=\F_{q^2}(u,v)$ defined by $u^{q+1}=v^q+v$ is the famous Hermitian maximal curve with the largest possible genus.
		It is easy to see that $Y_2\subseteq H$. Moreover, $H/Y_2$ is an Artin--Schreier extension with explicit equation $v^2+v=y$ of degree $2$.
Let $\sigma$ be an $\F_{q^2}$-automorphism group of $H$ determined by $\sigma(u)=u$ and $\sigma(v)=v+1$. It is easy to verify that $Y_2$ is the fixed subfield of $H$ under the group generated by $\sigma$. Thus, $Y_2$ is a maximal function field.

		\item
Since $Y_2$ is maximal, it has exactly $q^2+1+2qg(Y_2)=1+q^3/2$ rational places. From item (1), the infinity place $P_\infty$ is the unique place lying over $\infty$.
By the fundamental equation, the zero of $x-a$ in $\F_{q^2}(x)$ has at most $q/2$ extensions in $Y_2$ for each $a\in \F_{q^2}$.
Thus, the zero of $x-a$ splits completely into $q/2$ rational places of $Y_2$, i.e., for any $a\in \F_{q^2}$, there are exactly $q/2$ elements $b$ in $\F_{q^2}$ such that $h(b)=a^{q+1}$.
		
	\end{enumerate}
\end{proof}

\section[The automorphism group of the AT curve]{The automorphism group of the Abd\'on--Torres curve}\label{sec:3}
In this section, we determine the automorphism group of a maximal function field with the second largest genus over the finite field with even characteristic via the group action of automorphism group on the set of all rational places.
Firstly, we compute the principal divisors of chosen functions and determine Weierstrass semigroups of rational places. Secondly, we determine the stabilizer of the infinite place $P_\infty$ via Riemann--Roch spaces. Finally, we determine the orbit of the group action and the full automorphism group of the Abd\'on--Torres curve.

\subsection{Weierstrass semigroups of \texorpdfstring{$P_\infty$}{P\_\textbackslash infty} and the stabilizer of \texorpdfstring{$P_\infty$}{P\_\textbackslash infty} }
Let $F/\F_{q^2}$ be a function field with genus $g (F)$ and $P$ be a rational place of $F$. For any non-negative integer $m$, if there exists an element $x\in F$ with $(x)_\infty=mP$, then $m$ is called a pole number of $P$.  The Weierstrass semigroup of $P$ consists of all pole numbers of $P$ and it is denoted by
$$H(P):=\{n\in \mathbb{N}: \exists x\in F \text{ such that } (x)_\infty=mP\}.$$
By the Weierstrass Gap Theorem \cite[Theorem 1.6.8]{Sti2009}, $H(P)$ is a numerical semigroup with genus $g(F)$ under the addition operation.

Let $h(y):=y^{q/2} + y^{q/4} + \dots + y\in \F_{q^2}[y]$ and let  $\Omega$ be the subset of $\F_{q^2}$ defined by $$\Omega=\{c\in \F_{q^2}: h(c)=0\}.$$
It is clear that $$h(y)=y^{q/2} + y^{q/4} + \dots + y=\prod_{c\in \Omega} (y-c).$$
From the Dedekind--Kummer Theorem \cite[Theorem 3.3.7]{Sti2009}, the principal divisor of $x$ in $Y_2$ is given by
$$(x)=-\frac{q}{2} P_\infty+\sum _{b\in \Omega} P_{0,b},$$
and the principal divisor of $y$ in $Y_2$ is given by
$$(y)=(q+1)(P_{0,0}-P_\infty).$$
It is clear that $q/2,q+1\in H(P_\infty)$.
The numerical semigroup $\langle q/2, q+1\rangle $ generated by $q/2$ and $q+1$ has genus $(q/2-1)(q+1-1)/2=q(q-2)/4=g({Y_2})$.
Hence, the Weierstrass semigroup of $P_\infty$ is $H(P_\infty)=\langle q/2, q+1\rangle$.
Furthermore, for any $m\in \mathbb{N}$, the Riemann--Roch space $\mathcal{L}(mP_\infty)$ has a $\mathbb{F}_{q^2}$-basis
$$\{x^iy^j: i\ge 0,\ 0\le j\le q/2-1,\ i\cdot q/2+j(q+1)\le m\}. $$

In this subsection, we determine the stabilizer defined by $G(P_\infty)=\{\sigma\in \Aut(Y_2/ \F_{q^2}): \sigma(P_\infty)=P_\infty\}$ of $P_\infty$ under the group action of automorphism group $\Aut(Y_2/ \F_{q^2})$ on the set of rational places of the Abd\'on--Torres curve.

\begin{proposition}\label{prop:stabilizer}
	If $q=2^n\ge 4$, then the stabilizer $G(P_\infty)$ of the infinite place $P_\infty$ of the Abd\'on--Torres curve $Y_2$ consists of all $\F_{q^2}$-automorphisms $\sigma$ of $Y_2$ determined by
	\[\begin{cases}
		\sigma(x)=  ax + b \\
		\sigma(y)= y + (ab^q)^2 x^2 + ab^q x + c
	\end{cases},\]
where $a,b,c \in \F_{q^2}, \ a^{q+1} = 1,$ and $b^{q+1} = h(c)$. \label{prop:struct_stab}
\end{proposition}

\begin{proof}
Since $H(P_\infty) = \langle q/2, q+1 \rangle $, we have $ \dim \mcal{L}(\frac{q}{2}P_\infty) $ $= 2$ and $\dim \mcal{L}((q+1) P_\infty)$ $ = 4$. More precisely,
	\[\mcal{L} \left(\frac q 2 P_\infty\right) = \F_{q^2} \oplus \F_{q^2} x, \quad \mcal{L}((q+1)P_\infty) = \F_{q^2} \oplus \F_{q^2} x \oplus \F_{q^2} x^2 \oplus \F_{q^2} y.\]
For any $\F_{q^2}$-automorphism $\sigma\in G(P_\infty)$, we have $\sigma(P_\infty)=P_\infty$ and $\sigma(\mcal{L}(nP_\infty))=\mcal{L}(\sigma(nP_\infty))=\mcal{L}(nP_\infty)$.
In particular, we have
	\[\begin{cases}
		\sigma(x)=ax + b, \\
		\sigma(y)=Ey + Ax^2 + Bx + c,\end{cases}\]
where $a, b, c, E, A, B\in \F_{q^2}$ and $ a, E \neq 0.$
Since the Abd\'on--Torres curve $Y_2$ is given by $Y_2=\F_{q^2}(x,y)$ with $ x^{q+1} = h(y)$, we obtain $$ (\sigma (x))^{q+1} = h (\sigma (y)).$$
That is to say
\[a^{q+1} x^{q+1} + a ^qb  x^q + ab^q  x + b^{q+1} = \sum_{j = 1}^{n} E^{q/2^{j}}y^{q/2^{j}} +  \sum_{j =1}^{n}\left(A^{q/2^{j}}x^{q/2^{j-1}}+B^{q/2^{j}}x^{q/2^{j}}\right) +h(c).\]
Let $\nu_{P_\infty}$ be the normalized discrete valuation with respect to $P_\infty$.
Since $\nu_{P_\infty}(x)=-q/2$ and $\nu_{P_\infty}(y)=-(q+1)$,
it is clear that $\nu_{P_\infty}(x^{q+1})=\nu_{P_\infty}(y^{q/2})<\nu_{P_\infty}(y^{q/4})<\nu_{P_\infty}(x^{q/2})<\nu_{P_\infty}(y^{q/8})<\nu_{P_\infty}(x^{q/4})<\cdots <\nu_{P_\infty}(y)<\nu_{P_\infty}(x^2)<\nu_{P_\infty}(x)<\nu_{P_\infty}(1)=0$.
By the Strict Triangle Inequality \cite[Lemma 1.1.11]{Sti2009}, both sides of the above equation have discrete valuation $-(q+1)q/2$.

Combining with $ x^{q+1} = h(y)$, we have $  a^{q+1}=E^{q/2}$ and
\[
\begin{split}
	(a ^qb-A^{q/2})  x^q +  \sum_{j =1}^{n-1}\left(A^{q/2^{j+1}}+B^{q/2^{j}}\right) x^{q/2^{j}}+ (ab^q-B)  x + b^{q+1} \\
	= \sum_{j = 2}^{n} (E^{q/2^{j}}-E^{q/2})y^{q/2^{j}}  +h(c).
\end{split}
\]
These monomials have pairwise distinct discrete valuations. Hence, we obtain
	\[\begin{cases}
		A^{q/2} =a ^qb, \\
		A^{q/2^{j+1}}+B^{q/2^{j}}=0 \text{ for } j = 1,2, \dots, n-1,\\
B=ab^q,\\
		E^{q/2^{j}}-E^{q/2} = 0 \text{ for } j = 2,3, \dots, n,\\
		b^{q+1} = h(c). \\
	\end{cases}\]
If $n\ge 2$, then $E^{q/2} = E^{q/4}$. Together with $E^{q/2} = a^{q+1}$, we deduce that $a^{q+1}=1$ and $E=1$.
Moreover, we have $A=(a ^qb)^{2q}=a^2b^{2q}=(ab^q)^2$.
Conversely, it is easy to verify that such a mapping $\sigma$ determines an $\F_{q^2}$-automorphism of $Y_2$.
The proof is complete.
\end{proof}

\subsection{The Weierstrass semigroups of other rational places \texorpdfstring{$P_{(a,b)}$}{P\_\{(a,b)\}}}
In this subsection, we determine the Weierstrass semigroups of rational places other than $P_\infty$ of the Abd\'on--Torres function field $Y_2$.

Firstly, let us consider the Weierstrass semigroup of the rational place $P_{0,b}$ for any $b\in \F_{q^2}$ with $h(b)=0$.
For any $a\in \F_{q^2}$, there exist $q/2$ elements $b+c\in \F_{q^2}$ with $c\in \Omega$ such that $h(b+c)=a^{q+1}$.
Thus, the principal divisor of $x-a$ in $Y_2$ is given by
$$(x-a)= -\frac q 2 P_\infty + \sum_{c \in \Omega} P_{a,b+c}.$$
For any $b\in \Omega$, the principal divisor of $y-b$ in $Y_2$ is given by
$$(y - b) = (q+1) (P_{0,b} - P_\infty).$$
By taking the reciprocal of the above equation, it is easy to see that $q+1\in H(P_{0,b})$. Moreover, we have
		\[
		 	\left( \frac {x} {y-b} \right) = \left(\frac q2 + 1 \right)P_\infty + \sum_{b\neq c\in \Omega  } P_{0,c} - q P_{0,b},
		\]
		and
		\[
			\left( \frac {x^2} {y - b} \right) = P_\infty + 2 \sum_{b\neq c\in \Omega } P_{0,c} -(q-1) P_{0,b}.
		\]
		Thus, we obtain $q, q - 1 \in H(P_{0,b})$.
By \cite[Lemma 3.4]{CKT1999}, the numerical semigroup generated by $q-1,q$ and $q+1$ has genus
$$g(\langle q-1,q,q+1\rangle)=\frac{(q-2)q}{4}=g(Y_2).$$
Hence, the Weierstrass semigroup of $P_{0,b}$ is $H(P_{0,b})=\langle q-1,q,q+1\rangle$.

\begin{lemma}\label{lem:0b}
Let $P_{0,b}$ be the common zero of $x$ and $y-b$ with $h(b)=0$.
The Weierstrass semigroup of $P_{0,b}$ is given by
$$H(P_{0,b})=\langle q-1,q,q+1 \rangle.$$
\end{lemma}

Now let us consider the Weierstrass semigroup of the rational place $P_{a,b}$ for any $a\in \F_{q^2}^*$ and $b\in \F_{q^2}$ with $h(b)=a^{q+1}$.
Let $\zeta$ be a primitive $(q+1)$-th root of unity.
For any $a\in \F_{q^2}$ and any $b\in \F_{q^2}\setminus \Omega$ with $a^{q+1}=h(b)$, the principal divisor of $y-b$ in $Y_2$ is given by
$$(y - b) = - (q+1) P_\infty + \sum_{j  =0}^q P_{\zeta^j a, b}.$$

\begin{lemma}\label{lem:tangentline}
Let $q=2^n\ge 4$ be a prime power.
Let $a, b$ be nonzero elements in $\F_{q^2}$ such that $a^{q+1}=h(b)$.
Let $t_{a,b} := (y - b) - a^q (x - a)$ be the tangent line function at the rational place $P_{a,b}$.
Then there exists an effective divisor $E_{a,b}$ of $Y_2$ with $\supp(E_{a,b}) \cap \{ P_{a,b},P_\infty\}=\varnothing$ such that
		$$(t _{a,b}) = 2 P_{a,b} + E_{a,b}  - (q+1) P_\infty. $$
\end{lemma}
\begin{proof}
By the Strict Triangle Inequality \cite[Lemma 1.1.11]{Sti2009}, we have
\[\nu_{P_\infty} (t_{a,b}) = \min \{\nu_{P_\infty} (y - b), \nu_{P_\infty} (x - a)\} = -(q+1).\]
Since $x-a, y - b \in \mcal{L}((q+1) P_\infty)$, we have $t_{a,b} \in \mcal{L}((q+1)P_\infty)$ and $$(t _{a,b})+(q+1) P_\infty\ge 0.$$
It remains to prove that $\nu_{P_{a,b}} (t_{a,b}) = 2$.
It is easy to see that $t = (x-a)/a$ is a prime element of $P_{a,b}$.
Since $x^{q+1} = h(y)$ and $a^{q+1} = h(b)$, we get
\[h(y-b)= x^{q+1}-a^{q+1}=(x-a)^{q+1}+a(x-a)^q+a^q (x-a)=a^{q+1}(t^{q+1}+t^q+t).\]
Assume that $y-b=b_1t+b_2t^2+\cdots\in \F_{q^2}[[t]]$, it is easy to verify that $b_1=a^{q+1}$ and $b_1^2+b_2=0$ by comparing the coefficients.
Thus, we obtain that $$y-b=a^{q+1}t+a^{2(q+1)}t^2+O(t^3),$$
where $O(t^m)$ stands for the terms with degree no less than $m$ in the formal power series ring $\F_{q^2}[[t]]$.
Hence, it is clear that $$t_{a,b} = (y - b) - a^q (x - a) =  a^{2q+2} t^2 + O(t^3).$$
The above local expansion shows that $\nu_{P_{a,b}} (t_{a,b}) = 2$.
This completes the proof.
\end{proof}

By the equation (10.8) in Section 10.2 of \cite{CurveBook}, since $Y_2$ is maximal, for a given rational place $P_\infty$, we know that each place $P$ satisfies $(q + 1) P_\infty \sim q P + \Fr(P)$, where $\Fr $ is the map on divisors extended from the Frobenius automorphism $z \mapsto z^{q^2}$ of the finite field $\F_{q^2}$.
Specially, $\Fr(P) = P$ when $P$ is rational on $Y_2$. Or from the theory of abelian variety \cite[Lemma 1]{RS94}, there exists a function $f_{a,b}$ for each rational place $P_{a,b} \in \mathbb{P}^{(1)} _{Y_2}$ such that
\[(f_{a,b}) = (q+1) (P_{a,b} - P_\infty).\]

\begin{lemma}\label{lem:ab}
Let $P_{a,b}$ be the common zero of $x-a$ and $y-b$ with $a^{q+1}=h(b)\neq 0$.
The Weierstrass semigroup of $P_{a,b}$ is the numerical semigroup generated by $q-1,q$ and $q+1$, i.e.,
$$H(P_{a,b})=\langle q-1,q,q+1 \rangle.$$
\end{lemma}
\begin{proof}
From the principal divisor of $f_{a,b}$, the principal divisor of
 $1/f_{a,b}$ gives that $q+1 \in H(P_{a,b})$.
Moreover, it is easy to see that
		\[\left(
				\frac {x - a} {f_{a,b}}
			\right)= \left(\frac q2 + 1\right) P_\infty + \sum_{0\neq c \in \Omega} P_{a,b + c} - q P_{a,b}\]
		and $q \in H(P_{(a,b)})$. By Lemma \ref{lem:tangentline}, we have
		\[
			\left( \frac {t_{a,b}} {f_{a,b}} \right)= E_{a,b} - (q-1) P_{a,b}
		\]
		and hence $q-1 \in H(P_{a,b})$.
The numerical semigroup $\langle q-1, q, q+1\rangle$ generated by $q-1,q$ and $q+1$ is a subset of $H(P_{a,b})$. Furthermore, the semigroup $\langle q-1, q, q+1\rangle$ has genus $ {(q-1)q/4} = g({Y_2})$ by \cite[Lemma 3.4]{CKT1999}. Hence, $ H(P_{a,b}) = \langle q - 1, q, q+1\rangle$.
\end{proof}

We summarize the main results of this subsection.
\begin{proposition}\label{prop:semigroup}
The Weierstrass semigroups of the rational places of the Abd\'on--Torres curve $Y_2/\F_{q^2}$ can be determined explicitly as follows:
\begin{itemize}
\item[(1)] $H(P_\infty) = \langle q/2, q+1 \rangle $;
\item[(1)] $H(P_{a,b}) = \langle q-1, q,q+1\rangle $ for any other rational place $P_{a,b}$ of $Y_2$.
\end{itemize}
\end{proposition}
\begin{proof}
This proposition follows from Lemmas \ref{lem:0b} and \ref{lem:ab}.
\end{proof}

\subsection{The automorphism group of the Abd\'on--Torres curve}
Next, we need to determine the orbit of $P_{\infty}$ under the group action of automorphism group $\Aut(Y_2/ \F_{q^2})$ on the set of all rational places.

\begin{proposition}\label{prop:orbit}
If $q=2^n\ge 4$, then the orbit $\mathcal{O}({P_\infty})$ of infinity place $P_{\infty}$ of $Y_2$ under the group action of automorphism group $\Aut(Y_2/ \F_{q^2})$ is $$\mathcal{O}({P_\infty})=\{\sigma(P_{\infty}): \sigma\in \Aut(Y_2/ \F_{q^2})\}=\{P_{\infty}\}.$$
\end{proposition}
\begin{proof}
Suppose that there exists an automorphism $\sigma\in \Aut(Y_2/ \F_{q^2})$ such that $\sigma(P_\infty)=P_{a,b}$ for some rational place $P_{a,b}$.
The principal divisor of $x-a$ in $Y_2$ is given by
$$(x-a)= -\frac q 2 P_\infty + \sum_{c \in \Omega} P_{a,b+c}.$$
Since $\nu_{\sigma(P)}(\sigma(z))=\nu_P(z)$ for any $z\in Y_2$ and $P\in \mathbb{P}_{Y_2}$ by \cite[Lemma 3.5.2]{Sti2009}, we have
$$(\sigma(x-a))= -\frac q 2 \sigma(P_\infty) + \sum_{c \in \Omega} \sigma(P_{a,b+c}).$$
Hence, we have $q/2\in H(P_{a,b})$, which is a contradiction with Proposition \ref{prop:semigroup}.
\end{proof}

\begin{theorem}\label{thm:mainresult}
Let $q=2^n\ge 4$ be a prime power and $Y_2/\F_{q^2}$ be the maximal function field $\F_{q^2}(x,y)$ defined by
$$y^{q/2}+y^{y/4}+\cdots+y=x^{q+1}.$$
Then the automorphism group of $Y_2$ over $\F_{q^2}$ is exactly the stabilizer of the infinity place $P_\infty$, i.e.,
 $$\Aut(Y_2/ \F_{q^2})=G(P_\infty).$$
\end{theorem}
\begin{proof}
By the orbit-stabilizer theorem and Proposition \ref{prop:orbit}, the order of automorphism group of $Y_2$ over $\F_{q^2}$ is
$$|\Aut(Y_2/ \F_{q^2})|=|G(P_\infty)|\cdot |\mathcal{O}({P_\infty})|=|G(P_\infty)|.$$
This completes the proof.
\end{proof}

Let $C$ be the cyclic subgroup of $\Aut(Y_2/ \F_{q^2})$ consisting all automorphisms $\sigma$ determined by $\sigma(x)=ax$ and $\sigma(y)=y$ for any $a\in \F_{q^2}$ with $a^{q+1}=1$, i.e.,
$$C=\{\sigma\in \Aut(Y_2/ \F_{q^2}): \sigma(x)=ax, \sigma(y)=y, a^{q+1}=1\}.$$
Let $N$ be the subgroup of $\Aut(Y_2/ \F_{q^2})$ consisting all automorphisms $\sigma$ determined by $\sigma(x)=x+b$ and $\sigma(y)=y+b^{2q}x^2+b^qx+c$ for any $b,c\in \F_{q^2}$ with $b^{q+1}=h(c)$, i.e.,
$$N=\{\sigma\in \Aut(Y_2/ \F_{q^2}): \sigma(x)=x+b, \sigma(y)=y+b^{2q}x^2+b^qx+c, b^{q+1}=h(c)\}.$$

\begin{proposition}\label{prop:semiproduct}
	The automorphism group of $Y_2$ over $\F_{q^2}$ consists of $(q+1)q^3/2$ $\F_{q^2}$-automorphisms of $Y_2$ given by
	\[
	\begin{cases}
		\sigma(x)= ax + b \\
		\sigma(y)= y + (ab^q)^2 x^2 + ab^q x + c
	\end{cases},
	\]
where $ a,b,c \in \F_{q^2}, \ a^{q+1} = 1$ and $b^{q+1} = h(c)$.
Moreover, the automorphism group $\Aut(Y_2/ \F_{q^2})$ is a semi-product of a cyclic subgroup with order $q+1$ and an elementary $2$-group with order $q^3/2$, i.e.,  $$\Aut(Y_2/ \F_{q^2})= C \ltimes N.$$
\end{proposition}
\begin{proof}
Let $\sigma_i$ be automorphisms of $Y_2$ determined by $\sigma_i(x)=a_ix+b_i$ and $\sigma_i(y)=y + (a_ib_i^q)^2 x^2 + a_ib_i^q x + c_i$ for $i=1,2$.
It is easy to verify that the composition $\sigma_2\cdot \sigma_1$ of $\sigma_1$ and $\sigma_2$ can be explicitly given by $(\sigma_2\cdot \sigma_1)(x)=a_1a_2x+a_1b_2+b_1$ and $(\sigma_2\cdot \sigma_1)(y)=y + (a_2b_2^q+a_1a_2b_1^q)^2 x^2 + (a_2b_2^q+a_1a_2b_1^q)x + (a_1b_2b_1^q)^2 + a_1 b_2 b_1^q + c_1 + c_2$.
If we identify each $\sigma_i\in \Aut(Y_2/ \F_{q^2})$ in the above form with a triple $ [a_i,b_i,c_i]$, then the composition $\sigma_2\cdot \sigma_1$ can be characterized as
$$[a_2,b_2,c_2]\cdot [a_1,b_1,c_1]=[a_1a_2, a_1 b_2 + b_1, (a_1b_2b_1^q)^2 + a_1 b_2 b_1^q + c_1 + c_2].$$
Hence, the mapping $\varphi: \Aut(Y_2/ \F_{q^2})\rightarrow C$ determined by $[a,b,c]\mapsto [a,0,0]$ is a group homomorphism.
It is easy to see that $\varphi$ is surjective and its kernel is $N$.
Therefore, $N$ is a normal subgroup of $\Aut(Y_2/ \F_{q^2})$.
For any $[a,b,c]\in \Aut(Y_2/ \F_{q^2})$, it is clear that $[a,b,c]=[a,0,0]\cdot [1,b,c]\in CN$, i.e., $\Aut(Y_2/ \F_{q^2})=CN$.
By \cite[Theorem 5.12]{DF}, we have $\Aut(Y_2/ \F_{q^2})= C \ltimes N.$
\end{proof}

\begin{remark}
	Note that $g(Y_2) = q(q-2)/4$. It is easy to see that
	\[
		|\Aut (Y_2 / \F_{q^2})| = \frac {q^3(q+1)}2 > \frac {q^2 (q-2)^2} 2 = 8g(Y_2)^2.
	\]
	This indicates that the group $\Aut (Y_2 / \F_{q^2})$ is rather large compared with its genus $g(Y_2)$.
\end{remark}

\section{Fixed subfields of \texorpdfstring{$Y_2$}{Y\_ 2} with explicit defining equations}\label{sec:4}

In this section, we investigate maximal subfields of $Y_2$ that are stable under the action of subgroups of $\Aut (Y_2 / \F_{q^2})$. We first consider the ramification behavior of fixed subfields of $Y_2$, and we provide an explicit family of maximal function fields by considering its fixed subfields.

\subsection{The ramification behavior and its conductor}

In this subsection, we investigate the ramification behavior of the field extension $Y_2/ Y_2^{\Aut(Y_2/\F_{q^2})}$, and use it to compute the conductor of $Y_2 / \F_{q^2}(x)$.

\begin{proposition}\label{prop:ram_Aut_Y2}
Let $A$ denote the automorphism group $\Aut(Y_2/\F_{q^2})$ and $Z=Y_2^A$ be the fixed subfield of $Y_2$ with respect to $A$.
Then the fixed subfield $Z$ is a rational function field.
There are exactly two places of $Z$ ramified in the extension $Y_2/Z$.
One of the ramified places is $ R_\infty := P_\infty \cap Z$: it ramifies totally and wildly with ramification index
	$e(P_\infty|R_\infty)= q^3 (q+1)/2$
	and different exponent
	\[
	d(P_\infty|R_\infty) = \frac {1}{2} \left(q^4 + 2q^3 +{q^2}- 2q -4 \right).
	\]
The other one is a rational place $R$ of $Z$ that lies under any rational place $P \neq P_\infty$ in $Y_2$ and tamely ramifies with ramification index $ e(P|R) = q+1$ and different exponent $d(P|R)=q$. The conjugates of $P$ under $A$ are exactly all rational places of $Y_2$ that is not $P_\infty$.
\end{proposition}

\begin{proof}
By Theorem \ref{thm:mainresult}, the automorphism group $A$ is the decomposition group of $Y_2/ Y_2^A$ at $P_\infty$ and it follows that the ramification index of $P_\infty | R_\infty$ is $e(P_\infty | R_\infty) = q^3 (q+1)/2$.
To determine the different exponent, we use the Hilbert's ramification theorem \cite[Theorem 3.8.7]{Sti2009} as follows.
Let $t = x^2/y$ be a prime element of the infinity place $P_\infty$ by Proposition \ref{prop:semigroup}.
For any $\F_{q^2}$-automorphism  $\sigma$ of $Y_2$ given by $\sigma (x) = ax + b$ and $\sigma (y) = y + (ab^qx)^2 + ab^qx + c$, we have
	\begin{align*}
		\nu_{P_\infty}(\sigma(t) - t) &= \nu_{P_\infty} \left(\frac {(ax+b)^2} {y + (ab^qx)^2 + ab^q x + c} - \frac {x^2}y\right)  \\
		&= \nu_{P_\infty}  (y(a^2x^2 + b^2) - x^2 y + a^2 b^{2q} x^4 + ab^q x^3 + cx^2) - 2 \nu_{P_\infty} (y)\\
		&= \nu_{P_\infty} ((a^2 + 1) x^2 y + a^2 b^{2q } x^4 + ab^q x^3 + cx^2 ) - (2q+2) \\
		&= \begin{cases}
			1, & a \neq 1;\\
			2, & a = 1, b \neq 0; \\
			q + 2, & a =1 , b= 0, c \neq 0.
		\end{cases}
	\end{align*}
The different exponent $d (P_\infty|R_\infty)$ can be computed as
\begin{align*}
	d(P_\infty | R_\infty) &= \sum_{\sigma \in G \setminus \{\operatorname{id}\} } \nu_{P_\infty} (\sigma (t) - t)= 1 \cdot \frac {q^4}2 + 2 \cdot (q^2 - 1)\frac q2 + \left(  q + 2\right)\cdot \left(\frac q2 - 1 \right)\\
&= \frac 12(q^4 + 2q^3 + q^2 - 2q - 4).
\end{align*}
Since the genus of $Y_2$ is $q(q-2)/4$ by Proposition \ref{prop:ram_Y2}, the Hurwitz genus formula yields
\[\frac{q (q - 2)}2 - 2 = |A| \cdot (2 g ({Z}) - 2 ) + \deg \Diff (Y_2/ Z).\]
Since $\deg {\Diff (Y_2/ Z)}=d(P_\infty|R_\infty)+  \sum_{R_\infty\neq R\in \mathbb{P}_Z} \sum_{P|R} d(P|R)\deg(P)>2g({Y_2})-2$, it is easy to see that $g(Z)=0$ and
$\ \sum_{R_\infty\neq R\in \mathbb{P}_Z} \sum_{P|R}d(P|R)\deg(P)=q^4/2$.

Since $Y_2/ Z$ is a Galois extension with Galois group $A$, for each $ R \in \mathbb{P}_{Z} $ we denote $e(R), d(R), f(R)$ as $e(P|R), d(P|R), f(P|R)$  respectively for any place $ P \in \mathbb P_{Y_2}$  lying above $R$.
By combining the fundamental equation, we have
 $$ \sum_{R_\infty\neq R\in \mathbb{P}_Z}\sum_{P|R} d(P|R)\deg(P)=q^4/2=|A|\cdot \sum_{ R \neq R_\infty }\frac {d(R) } {e(R)} \deg R  =\frac {q}{q+1} |A|.$$
It follows that
\[\sum_{R_\infty\neq R\in \mathbb{P}_Z}\frac {d(R) } {e(R)} \deg R  = \frac {q}{q+1}.\]
By Dedekind's Different Theorem \cite[Theorem 3.5.1]{Sti2009}, the inequality $ d(R) \geq e(R) - 1$ holds true for all places $R$.
Hence, there exists exactly one rational place $R \neq R_\infty$ such that $ e(R) > 1$. Otherwise, the left-hand side of the above equation is at least $ 2 - 1 / e(R_1) - 1 / e(R_2) \geq 2 - 1/2 - 1/2 = 1$ or $(1-1/e(R))\cdot 2\ge 1$, which is a contradiction. Furthermore, it is easy to see that $e(R) = q+1$ and $d(R) = e(R) - 1 = q$, i.e., $R$ is tamely ramified in the extension $Y_2/Z$.
	
Let $P \in \mathbb{P}_{Y_2}$ be any place lying above $R$. We claim that $P$ must be a rational place.
If not, then there exists a rational place $ P_0 \neq P_\infty \in \mathbb{P}_{Y_2}$ such that $R_0 = P_0 \cap Z$ is not equal to $R$.
Since $P_0$ is rational, the relative degree of $P_0|R_0$ is $ f(P_0|R_0) = 1$. Since $ R_0 \notin \{ R_\infty, R\}$, we have $ e(P_0 | R_0) = 1$. Therefore, there exist at least $ |A| = (q+1)q^3 / 2 > q^3 /2 + 1$ rational places lying over $R_0$, which is impossible. Thus, all finite rational places of $Y_2$ lie above $R$.
\end{proof}

\begin{proposition}\label{prop:cond}
The conductor of $Y_2/\F_{q^2}(x)$ is $$\operatorname{Cond}(Y_2/\F_{q^2}(x))=(q+2)\cdot \infty.$$
\end{proposition}
\begin{proof}
By Proposition \ref{prop:ram_Y2}, the infinity place $\infty$ of $\F_{q^2}(x)$ is the unique ramified place in the extension $Y_2/\F_{q^2}(x)$.
Denote the place lying above $\infty$ by $P_\infty $.
The Galois group of $Y_2/\F_{q^2}(x)$ is $\text{Gal}(Y_2/\F_{q^2}(x))=\{\sigma\in \Aut(Y_2/\F_{q^2}): \sigma(y)=y+c \text{ with } h(c)=0\}$.
From the proof of Proposition \ref{prop:ram_Aut_Y2}, the $i$-th higher ramification group $G_i(P_\infty|\infty)$ of $P_\infty|\infty$ are given by
$$G_0(P_\infty|\infty)=G_1(P_\infty|\infty)=\cdots=G_{q+1}(P_\infty|\infty)=\text{Gal}(Y_2/\F_{q^2}(x))$$ and
$G_{q+2}(P_\infty|\infty)=\{{\rm id}\}$. The least positive integer $a(P_\infty|\infty)$ such that the higher ramification groups of $P_\infty|\infty$ are trivial is $a(P_\infty|\infty)=q+2$.
Thus, the conductor exponent of $P_\infty|\infty$ is given by
$$c(P_\infty|\infty)= \frac{d(P_\infty|\infty)+a(P_\infty|\infty)}{e(P_\infty|\infty)}=\frac{(q+2)(q/2-1)+(q+2)}{q/2}=q+2. $$
\end{proof}

\begin{remark}
By Proposition \ref{prop:ram_Y2}, the zero of $x-a$ in $\F_{q^2}(x)$ splits completely in $Y_2$ for each $a\in \F_{q^2}$.
By Proposition \ref{prop:cond}, the conductor of $Y_2/\F_{q^2}(x)$ is $\operatorname{Cond}(Y_2/\F_{q^2}(x))=(q+2)\cdot \infty.$
A natural question is whether $Y_2$ is the maximal abelian extension of $\F_{q^2}(x)$ such that all zeros of $x-a$ split completely and conductor upper bounded by $(q+2)\cdot \infty$ or not.
In fact, the Hermitian function field is isomorphic to the largest ray class field
of conductor $(q+2)\cdot \infty$ in which all places of degree one different from $\infty$ of $\F_{q^2}(x)$ split completely \cite{99La}.
\end{remark}

Based on Proposition \ref{prop:ram_Aut_Y2} and the fact that any automorphism in $A$ can be identified with a triple $[a,b,c]$, we  can establish a similar result concerning the genera of Galois subfields of $Y_2$ by mimicking the method given in \cite[Sections 3 and 4]{GSX2000}.
The proof is rather long and similar to \cite[Theorem 4.4]{GSX2000}. Hence, we omit the details.

\begin{theorem} \label{thm:genus_Gal_subfld}
Let $G$ be a subgroup of $\Aut(Y_2/\F_{q^2})$ with order $m \cdot 2^{\alpha + \beta}$, where $2^\alpha=|\{b\in \mathbb{F}_{q^2}: \exists c \text{ such that } [1,b,c]\in G\}|$ and $2^ \beta=|\{c\in \mathbb{F}_{q^2}: [1,0,c]\in G\}|$.
Then the fixed subfield $Y_2^G$ is a maximal function field with genus
\[g = g(Y_2^G) = \frac {2^{n} - 2^\alpha (m-1)} {2^{\alpha + \beta + 1} m} (2^{n-1} - 2^\beta).\]
\end{theorem}

\subsection{Defining equations of a family of Galois subfields}
In this subsection, we explicitly determine the defining equations of the subfields fixed by some subgroups of the full automorphism group of the maximal curve $Y_2$.

We start by refining the characterization on the structure of the additive group $$\Omega=\{c\in \F_{q^2}: h(c)=0\}.$$
Let $\tau: \F_{q^2} \to \F_{q^2}$ be the automorphism given by $\gamma \mapsto \gamma^2$.
Let $B = \F_2 [T]$ be the polynomial ring over $\F_2$. For $f (T)=\sum_{i=0}^m a_iT^i \in B$ and $z \in \F_{q^2}$, we define the ring action of $B$ on $\F_{q^2}$ by
\[z^{f(T)} := f(\tau) (z),\]
where $f(\tau)=\sum_{i=0}^m a_i\tau^i$ and $\tau^i$ is the $i$-th power of $\tau$ under the composition law.
Thus, we get a ring homomorphism $\phi: B \rightarrow {\End}_{\F_2} (\F_{q^2})$ with $f(T) \mapsto f(\tau)$.
Hence, $\F_{q^2}$ is a left $B$-module under this action.
For $q=2^n$ and $H(T) = (T^n -1) / (T-1)$, the group $\Omega$ is just the kernel of $\phi^{-1}(h)(\tau)=H(\tau)$.
If $p(T) \in \F_2 [T]$ is a factor of $H(T)$, then ${\Ker}(p(\tau))$ is a $B$-submodule of $\Omega$.
The defining equations of the corresponding fixed subfields of these submodules can be described explicitly.

\begin{proposition}\label{prop:subfld_eqn}
Let $m$ be a positive integer with $m|(q+1)$, and $p (T) \in \F_2 [T]$ be a factor of $H(T)=(T^n -1) / (T-1)$.
Let $C_m$ be the unique subgroup of $C$ with order $m$, let ${\Ker}(p(\tau))$ be the submodule of $\Omega$ and let $N_p=\{[1,0,c]\colon c\in {\Ker}(p(\tau))\}$.
The subgroup $G=C_mN_p$ is an abelian subgroup of $A$.
For any subgroup $G =  C_m \times N_p \leq A$, the fixed subfield $Y_2^G$ of $Y_2$ with respect to $G$ is given by
\[Y_2^G = \F_{q^2} (w,z), \quad 	w ^{(q+1)/m} = z ^{H(T) / p(T)},\]
where $B $ also acts on $Y_2$ that naturally extends the action on $\F_{q^2}$. In particular, $ w = x^m$ and $z = y^{p(T)}$.
\end{proposition}

\begin{proof}
Let $\zeta$ be a primitive $(q+1)$-th root of unity in $\F_{q^2}^*$. Let $C_m$ be the subgroup of $A$ generated by $\xi = [\zeta ^{(q+1) / m}, 0, 0]$.
The automorphism $\xi$ preserves the functions $w$ and $z$, since $\xi (w) = x^{m} (\zeta ^{(q+1)/m})^m = w$ and $\xi(y)=y$. For any $c\in {\Ker}(p(\tau))$, the corresponding automorphism is given by $\eta = [1,0,c]$ with $c^{p(T)} = 0$.
The automorphism $\eta=[1,0,c]$ preserves $w$, since $\eta(x)=x$. Let $p(T) =\sum_{j = 0}^d b_j T^j \in \F_2[T]$. By the additivity of $\tau$, the automorphism $\eta$ preserves $z$ as follows:
	\begin{align*}
		\eta (z) &= \eta (y^{ p(T) }) = \eta \left(\sum_{j = 0}^d b_j \tau^j (y)\right) \\
		&= \eta \left( \sum_{j = 1}^d b_j y^{2^j}\right)
		= \sum_{j = 0}^d b_j (y + c)^{2^j} \\
		&= (y+c)^{p(T)}
		= y^{p(T) } + c ^{p(T)} = z+0\\
		&= z.
	\end{align*}
It is easy to verify that $\xi\cdot \eta=[\zeta ^{(q+1) / m},0,c]=\eta\cdot \xi$. Thus, $G=C_mN_p=C_m \times N_p $.
Since $G =  C_m \times N_p \leq A$, we have $\F_{q^2} (w, z) \subseteq Y_2^G$.
	
Let $F = \F_{q^2} (w, z) $ with $w ^{(q+1)/m} = z ^{H(T) / p(T)}$.
Let $E = \F_{q^2} (w, y)$. It is clear that $E = F(y)$ with $y^{p(T)} = z$ and
 $Y_2=E(x)$ with $x^m=w$.
Hence, $E/F$ is an Artin--Schreier extension of degree $2^d$ and $Y_2/E$ is a Kummer extension of degree $m$.
Therefore, the degree of extension $Y_2/F$ is $[Y_2:F]=[Y_2:E]\cdot [E:F]=2^dm$.
Moreover, the degree of extension $Y_2$ over $Y_2^G$ is $|G|=2^dm$ as well.
Thus, we have $Y_2^G=\F_{q^2} (w, z)$.
\end{proof}

\begin{theorem}\label{thm:5.2}
Let $m$ be a positive integer with $m|(q+1)$, and $p (T) \in \F_2 [T]$ be a factor of $H(T)=(T^n -1) / (T-1)$ with degree $d$.
The fixed subfield $F=\F_{q^2} (w,z)$ defined by
\[w ^{(q+1)/m} = z ^{H(T) / p(T)}\]
is a maximal function field over $\F_{q^2}$ with genus $$g(F)=\frac{q+1-m}{2m}(2^{n-1-d}-1).$$
\end{theorem}
\begin{proof} This result follows from Theorem \ref{thm:genus_Gal_subfld}. Since its proof is omitted,  we can compute it directly from its defining equation.
By Proposition \ref{prop:subfld_eqn}, the function field $F=\F_{q^2} (w,z)$ defined by
$$w ^{(q+1)/m} = z ^{H(T) / p(T)}$$ is maximal.
It is an Artin--Schreier extension of $\F_{q^2} (w)$.
Let $\infty$ be the infinity place of $\F_{q^2} (w)$.  Then we have
$$\nu_{\infty}(w ^{(q+1)/m})=-\frac{q+1}{m},$$
which is relatively prime to $p$. Let $P_\infty$ be the unique place of $F$ lying above $\infty$.
By the theory of Artin--Schreier extensions \cite[Theorem 3.7.8]{Sti2009}, the different exponent of $P_\infty|\infty$ is
$$d(P_\infty|\infty)=\left(\frac{q+1}{m}+1\right)(2^{n-1-d}-1).$$
From the Hurwitz genus formula, we have
$$2g(F)-2=2^{n-1-d}(0-2)+\left(\frac{q+1}{m}+1\right)(2^{n-1-d}-1).$$
This completes the proof.
\end{proof}

\begin{corollary}
Let $m$ be a positive integer with $m|(q+1)$. 
The fixed subfield $F=\F_{q^2} (w,z)$ defined by
\[  z ^{q/2}+z^{q/4}+\cdots +z=w ^{(q+1)/m}\]
is a maximal function field over $\F_{q^2}$ with genus $$g(F)=2^{n-1}(2^{n-1-d}-1).$$
\end{corollary}

\begin{corollary}
Let $p (T) \in \F_2 [T]$ be a factor of $H(T)=(T^n -1) / (T-1)$ with degree $d$.
The fixed subfield $F=\F_{q^2} (w,z)$ defined by
\[z ^{H(T) / p(T)}=w\]
is a maximal function field over $\F_{q^2}$ with genus $$g(F)=\frac{q+1-m}{2m}(2^{n-1}-1).$$
\end{corollary}

\section{Automorphism groups of algebraic geometry codes from $Y_2$}
In this section, we determine the automorphism group of the one-point algebraic geometry code from the Abd\'on--Torres function field $Y_2/\F_{q^2}$ and its infinite place $P_\infty$ as an application. Except for the trivial cases, we shall prove that the automorphism group of the one-point algebraic geometry code from $Y_2$ is its automorphism group $A = \Aut (Y_2 / \F_{q^2})$.

We begin by listing the elements of the finite field $\F_{q^2} $ as $\{x_1, x_2, \dots, x_{q^2}\}$. By the results in Section \ref{sec:3}, for each $x_j \in \F_{q^2}$, there exist $q/2$ elements $y_{j,1}, y_{j,2}, \dots, y_{j, q/2} \in \F_{q^2}$ such that the principal divisor of $x - x_j$ in $Y_2$ is given by
\[
	(x  -  x_j) = -\frac q 2 P_\infty + \sum _{\ell = 1}^{q/2} P_{x_j, y_{j,\ell}}.
\]
Let $ P_{x_j, y_{j,\ell}}$ be the common zero of $ x -x_j$ and $y - y_{j, \ell}$.
We write $P_{x_j, y_{j,\ell}}$ as $P_{j,\ell}$ for $j = 1, 2, \dots, q^2$ and $\ell = 1, 2, \dots, q/2$. These are all rational places of $Y_2$ which are different from the infinite place $P_\infty$. Let
\[\mcal D = \{
	P_{1,1}, P_{1,2}, \dots, P_{1, q/2}, P_{2,1}, \dots, P_{2,q/2}, \dots,P_{q^2,1}, \dots, P_{q^2,q/2}\}\]
be an ordered set or a list. For the list $\mcal D$ and $m \in \N$, we define the associated one-point algebraic geometry code as
$$C_m = C (\mcal D, mP_\infty)=\{(f(P))_{P\in \mcal D}: f\in \mcal{L}(mP_\infty)\}.$$
In the following, we shall determine the automorphism group of code $C_m$.

The useful result of Wesemeyer can be found from \cite[Theorem 5.12]{Wes98} which can cover many positive integers for $m\in \N$.

\begin{lemma}
	Let $\F$ be a finite field, and $F / \F$ be an algebraic function field with the following properties.
	\begin{itemize}
		\item[(1)] $F$ is not rational.
		\item[(2)] There exist a rational place $P_\infty$ of $F$ and two elements $x, y \in F$ satisfying that $(x)^F_\infty = AP_\infty$ and $(y)^F_\infty = BP_\infty$ for $A, B \in \mathbb{Z}_+$.
		\item[(3)] For each natural number $m$, the Riemann-Roch space $\mcal L(mP_\infty)$ has an $\F$-basis
		\[
		\{ x^j y^\ell \colon j \geq 0, 0\leq \ell \leq A - 1, jA + \ell B \leq m\}.
		\]
	\end{itemize}
	Suppose $J \subseteq \mbb P_F^{(1)} \setminus \{P_\infty \}$ with size $n = |J|$, $J$ is listed as $\{P_1, P_2,\dots, P_n\}$ and let $D=\sum_{j=1}^n P_j$. Let $g $ be the genus of $F$ and let $ \beta = \min \{A - 1, \max \{ r \in \N \colon y^r \in \mcal L(mP_\infty) \}\}$. Additionally assume $m \geq B > A$.
Let $C$ be the algebraic geometry code $C (J,mP_\infty)$.
If \[n > \max \left\{
	2 g +2, 2m, A \cdot \left( B + \frac {A-1}\beta \right), AB \left(1 + \frac {A - 1}{m - A +1}\right) \right\},\]
	then we have
	\[
	\Aut (C) \cong \Aut _{D, mP_\infty}(F/\F).
	\] \label{lem:Wesemeyer}
\end{lemma}

\begin{theorem}\label{prop:5.2}
Assume $q \geq 4$. For the one-point algebraic geometry code $C_m=C(\mathcal{D},mP_\infty)$ from the Abd\'on-Torres function field $Y_2/\F_{q^2}$, its automorphism group can be characterized as follows.
	\begin{enumerate}
		\item If $0 \leq m \leq q/2-1$ or $m \geq (q^3 + q^2 - 3q -2) /2$, then
		\[
		\Aut (C_m) \cong  S_{q^3/2},
		\]
		where $ S_{q^3/2}$ is the full symmetric group of a set of $q^3/2$ elements.
		\item If $ q/2 \leq m \leq q$ or $(q^3 + q^2 - 4q - 4)/2 \leq m \leq (q^ 3+ q^2 - 3q - 4)/2$, then
		\[
		\Aut (C_m) \cong \mrm {Aff}_1(\F_{q^2}) \rtimes ( S_{q/2}) ^{q^2},
		\]
		where $\mrm {Aff}_1 (\F_{q^2})$ is the group of affine linear transformation of $\F_{q^2}$, and $ ( S_{q/2}) ^{q^2}$ is the direct product of $q^2$ copies of the symmetric group $ S_{q/2}$.
		\item If $ q + 1 \leq m \leq q^3 / 4 - 1$ or $(q^3+2q^2 - 4q - 4)/4 \le m\le (q^3 + q^2 - 4q - 6)/2$,
		then
		 $$\Aut (C_m) \cong \Aut(Y_2/\F_{q^2}).$$
	
	\end{enumerate} \label{prop:aut_code_easy}
\end{theorem}

\begin{proof}
	First we note that $\Aut (C_m ^\perp)  = \Aut (C_m) $ from \cite{MS1977}. This comes from the fact that if  $\bm x \cdot \bm y = \bm x \bm y\trans = 0$ for row vectors $\bm x, \bm y \in \F_{q^2}^n$, then  $\sigma (\bm  x) \cdot \sigma (\bm y) = 0$ for each $\sigma \in S_n $. Thus, we shall compute the dual code $C_m^\perp$.
	Let $t = x^{q^2} -x = \prod _{j =1}^{q^2} (x  -x_j)\in Y_2$. It is easy to see that $\nu_{P}(t)=1$ for any rational place $P\in \mcal D$ and $(t)=\sum_{P \in \mcal D} P-q^3/2 \cdot P_\infty$. For any rational place $P\in \mcal D$, the differential $\eta:=dt/t$ satisfies $\nu_P(\eta)=-1$ and its residue at $P$ is $\text{res}_{P}(\eta)=-1$.
By  \cite[Theorem 3.4.6 and Theorem 4.3.2]{Sti2009}, we have
	\[
		(\mrm d t) = (-\mrm dx) = -2 (x)_\infty + \Diff (Y_2/ \F_{q^2} (x)) = \frac 12 (q^2 - 2q-4) P_\infty.
	\]
By \cite[Proposition 8.1.2]{Sti2009}, we can obtain the dual code
	\[
	C(\mcal D , mP_\infty)^{\perp} =  C\left(\mcal D, \sum_{P \in \mcal D} P - mP_\infty + (\mrm d t ) - (t)\right)=C(\mcal D , (M-m)P_\infty)
	\]	
	where $M: = (q^3 + q^2 - 2q - 4)/2$.

	\begin{enumerate}
		\item If $0 \leq m \leq q / 2 - 1$, then $\mcal L(mP_\infty) = \F_ {q^2}$, thus the generator matrix is the row vector $$\begin{pmatrix}
			1 & 1 & \cdots & 1
		\end{pmatrix}_{1 \times (q^3/2)} =: \mbf 1 _{q^3/2}, $$ where we use $ \mbf 1 _{n}$ to denote the row vector of length $n$ with all entries being $1$. Naturally, the automorphism group is $ S_{q^3/2}$. The same result holds for $ 0 \leq M - m \leq q/2 - 1$ by duality, which means the assertion is true for $(q^3 + q^2 - 3q - 2) / 2 \le m \le M$. Moreover, if $m\ge M+1$, then the code is the full vector space $\F_{q^2}^{q^3/2}$. Thus, its automorphism group is $  S_{q^3/2}$ as well.
		\item If $ q / 2 \leq m \leq q - 1$, then $\mcal  L(mP_\infty) = \F_{q^2}  \oplus \F_{q^2} x$. For $ m = q$, $ \mcal  L(qP_\infty) = \F_{q^2} \oplus \F_{q^2} x \oplus \F_{q^2} x^2$. Thus, the generator matrix of $C_m$ is either
		\[
		\begin{pmatrix}
			\mbf 1_{q/2} & \mbf  1 _{q/2} & \cdots & \mbf 1 _{q/2}\\
			x_1 \mbf 1_{q/2} & x_2\mbf  1 _{q/2} & \cdots & x_{q^2 } \mbf 1 _{q/2}
		\end{pmatrix},
		\]
		or
		\[
		\begin{pmatrix}
			\mbf 1_{q/2} & \mbf  1 _{q/2} & \cdots & \mbf 1 _{q/2}\\
			x_1 \mbf 1_{q/2} & x_2\mbf  1 _{q/2} & \cdots & x_{q^2 } \mbf 1 _{q/2}\\
			x_1^2 \mbf 1_{q/2} & x_2^2\mbf  1 _{q/2} & \cdots & x_{q^2 }^2 \mbf 1 _{q/2}
		\end{pmatrix}. 			
		\]
Hence, the automorphism group is $\mrm {Aff}_1 (\F_{q^2}) \rtimes ( S_{q/2} )^{q^2}$ from the proof given in \cite{KKO01} or \cite{Xing95b}. By observing the dual code, we can conclude the same statement for other values of $m$.
		\item For $m \geq q+1$, we can apply Lemma \ref{lem:Wesemeyer}. In this case, $A = q/2$, $B = q+1$ and $n = q^3/2$. It follows that $A<B\le m$ and  $2g + 2=(q^2 - 2q + 4)/2<n$. Moreover, we have
		\[
			A \left( B + \frac {A-1}\beta \right) \leq A (B+A - 1) = \frac 34 q^2 < n,
		\]
		and
		\begin{align*}
			AB \left(1 + \frac {A-1} {m+A-1}\right)&\leq AB \left( 1 + \frac A {q + 1 - A + 1}\right) \\
			&= \frac{q(q+1)}{2} \left(1 + \frac {q/2} {q/2 + 2}\right) < q(q+1)< n.
		\end{align*}
It remains to ensure that $2m<n$.
Thus, for $q+1 \leq m \leq n/2 - 1 = q^3 / 4 - 1$, the automorphism group of $C_m$ is $$\Aut(C_m)\cong \Aut _{D, mP_\infty}(F/\F_{q^2}) =G(P_\infty)= \Aut (F /\F_{q^2})$$ by Theorem \ref{thm:mainresult}. The proof is complete after we consider their dual codes. 	
	\end{enumerate}
\end{proof}

By Proposition \ref{prop:5.2}, it remains to determine the automorphism groups of $C_m$ for the gaps of $m$ with $ q^3 / 4 \leq m \leq (q^3 + 2q^2 - 4q - 6) / 4$.
In particular, we mainly focus on the pole numbers $m\in H(P_\infty)$. To treat these values of $m$, we adapt the main ideas of determining automorphism groups of Hermitian codes initiated in \cite{Xing95b}. A similar lemma of \cite[Lemma 3]{Xing95b} is given as follows.

\begin{lemma}\label{lem:5.3}
	Let $\mcal D'=\{P_{j,\ell}': 1\le j\le q^2, 1\le \ell\le q/2\}$ be a permutation of the list $ \mcal D=\{P_{j,\ell}: 1\le j\le q^2, 1\le \ell\le q/2\}$. If $C(\mcal D, mP_\infty) = C(\mcal D', mP_\infty)$ for $q/2\le m\le q^3/2$ and $m\in H(P_\infty)$, and there exists a subset $K \subseteq  {\mcal D}$ of cardinality $m$ and an element $w \in \mcal  L(m P_\infty)$ such that $K \subseteq Z(w)$, then the divisor
	\[
	-mP_\infty + \sum _{P \in K} P'
	\]
	is principal. \label{lem:crit_princ}
\end{lemma}

\begin{proof}
	Since $w \in \mcal L(m P_\infty)$, we obtain $mP_\infty - (w)_\infty \geq 0$. Hence,
	\[
		(w) = \sum_{P \in K} P - mP_\infty.
	\]
	From the fact $C(\mcal D, mP_\infty) = C(\mcal D', mP_\infty)$, we can find a function $w ' \in \mcal L(mP_\infty)$ such that $w ' (P') = w(P)$ for each $P\in \mcal D$. Therefore, $(w')_0 - \sum_{P\in K} P' \geq 0$ and hence
	\[
		(w') = \sum _{P \in K}P' - mP_\infty.
	\]
\end{proof}

For convenience, let us reformulate the above result as follows. Consider the natural group homomorphism
\[
\phi \colon \operatorname{Div} (Y_2) \to \operatorname{Pic}^0 (Y_2), \quad E \mapsto [E - (\deg E) P_\infty],
\]
then $E - (\deg E) P_\infty$ is principal if and only if $\phi (E) = 0$. The Lemma \ref{lem:5.3} means that  $\phi (\sum _{P \in K} P') = 0$ under the given conditions.


\begin{proposition}\label{prop:5.5} 
	For $\b \in {\F_{q^2}}$, let
\[z_\b := y + \b^2 x^2 + \b x,
	\] and for each $e \in {\F_{q^2}}$, let
	$T_e := \{\alpha \in {\F_{q^2}} \colon h(\alpha) = e\}.$ The following results hold true.
	\begin{enumerate}
		\item 	For every $\gamma \in T_{\b^{q+1}}$, the principal divisor of $z _\b - \gamma$ is
		\[
		(z _\b - \gamma) = (q+1) ( P_{\b^q, \gamma+ \b^{2q+2} + \b^{q+1}} - P_\infty ).
		\]
		\item Let $\mcal S := \{\delta^2 + \delta \colon \delta\in \F_{q^2} \}$. For each $\gamma \in \mcal S \setminus T_{\b^{q+1}}$, there exist distinct $q+1$ finite rational places $\tilde P_1, \tilde P_2, \dots, \tilde P_{q+1}$ of $Y_2$ such that
		\[
		(z_\b - \gamma)= -(q+1) P_\infty + \sum_{j = 1}^{q+1}\tilde P_j,
		\]
	\end{enumerate} \label{prop:princ_div}
\end{proposition}

\begin{proof}
Consider the extension $Y_2 / \F_{q^2}(z_\b)$. It is clear that $Y_2 = \F_{q^2} (x,y) = \F_{q^2} (x, z_\b)$.
Let $Q_\gamma$ be the zero of $z_\b - \gamma$ in $\F_{q^2} (z_\b)$ for each $\gamma \in \mcal S$.
Since $$x^{q+1} = h(y)= h(z_\b) +h(\b^2 x^2) +h(\b x)= h(z_\b) + \b^q x^q + \b x, $$
 $x $ is integral over the polynomial ring $ \mcal O_{Q_\gamma} [T]$ with minimal polynomial $\varphi (T)  = T^{q+1} + \b^q T^q + \b T + h(z_\b)$.
The polynomial $\varphi (T)$ can be written as
		\begin{align*}
			\varphi (T) &=
			(T + \b^{q}) ^{q+1} + h(z_\b-\gamma) + h(\gamma) + \b^{q+1}.
		\end{align*}
		From the Hilbert's 90 Theorem, there exists $\alpha \in \F_{q^2}$ such that $\b^{q+1} = \alpha^q + \alpha$.
For any $ \gamma=\delta^2 + \delta \in \mcal S$, there exists an element $\theta\in \F_{q^2}^*$ such that
 $$h(\gamma) + \b^{q+1} = h(\delta^2 + \delta) + \alpha^q + \alpha = (\alpha + \delta)^q + (\alpha + \delta)=\theta^{q+1}.$$
 Let $\zeta \in \F_{q^2}^*$ be a primitive $(q+1)$-th root of unity.
 The polynomial $\bar \varphi (T)$ as the image of $\varphi (T)$ under the natural projection $\mcal O_{Q_\gamma} [T] \to \kappa_{Q_\gamma} [T] $ is
		\[\bar \varphi (T) = (T+\b^q)^{q+1} - \theta^{q+1} = \prod_{j = 1}^{q+1} (T + \b^q + \zeta^j \theta).\]
\begin{itemize}
\item[(1)] For every $\gamma \in T_{\b^{q+1}}$, we have $\theta=0$. By the Dedekind--Kummer theorem \cite[Theorem 3.3.7]{Sti2009}, the zero of $ z_\b  - \gamma$ in   $\F_{q^2}(z_b)$ is totally ramified in $Y_2$ with ramification index $q+1$. In fact, it is the common zero of $x+\b^q$ and $y+ \gamma + \b^{2q+2} + \b^{q+1}$.
\item[(2)] By the Dedekind--Kummer theorem \cite[Theorem 3.3.7]{Sti2009}, there exist $q+1$ rational places $\tilde P_j$ of $Y_2$ lying above $Q_\gamma$ such that $(x + \b^q + \zeta^j \theta) \in \tilde P_{j}$ for $1\le j\le q+1$. 
 Combining with $(z_\b - \gamma)_\infty  = (q+1)P_\infty$, we conclude that
		\[
			(z_\b - \gamma) = -(q+1)P_\infty + \sum _{j =1}^{q+1} \tilde P_j.
		\]
\end{itemize}
\end{proof}

From the proof of Proposition \ref{prop:5.5}, we can determine these places $\tilde P_j$ explicitly.  In fact, $x + \b^q + \zeta^j \theta \in \tilde P_j$, i.e., $x(\tilde P_j) = \b^q + \zeta ^j \theta$. Since $(z _\b - \gamma) (\tilde P_j) = 0$,  we can obtain
\[
	y(\tilde P_j) = \gamma + (\b^{q+1} + \b\zeta ^j \theta)^2+(\b^{q+1} + \b\zeta ^j \theta).
\]
Thus, the place $\tilde P_j$ is the common zero of
$x+ \b^q + \zeta ^j \theta$ and $y +\gamma + (\b^{q+1} + \b\zeta ^j \theta)^2+(\b^{q+1} + \b\zeta ^j \theta).$


\begin{lemma}\label{lem:5.6}
Let $\mcal D'=\{P_{j,\ell}': 1\le j\le q^2, 1\le \ell \le q/2\}$ be a permutation of the list $ \mcal D=\{P_{j,\ell}: 1\le j\le q^2, 1\le \ell\le q/2\}$.
If $C(\mcal D, mP_\infty) = C(\mcal D', mP_\infty)$ for $q+1 \leq m \leq (q^3  - 3q - 2) / 2$, then
	\[
	\phi \left( \sum_{\ell = 1}^{q/2} P_{j,\ell}' \right) = 0
	\]
	\label{lem:key_AG_Code}
for each $1\le j\le q^2$.
\end{lemma}

\begin{proof}
	It will be sufficient to prove that
	\begin{equation}\label{5.6}
		\phi \left( \sum_{\ell = 1}^{q/2} P_{j, \ell} ' \right) = \phi \left( \sum_{\ell = 1}^{q/2} P_{j', \ell} ' \right)
	\end{equation}
for $1 \leq j < j' \leq q^2$.
Since \[
		(x^{q^2} - x) = -\frac {q^3 }2 P_\infty + \sum_{P \in \mcal D} P = -\frac {q^3 }2 P_\infty + \sum_{P \in \mcal D} P',
	\] we get $\phi (\sum_{P \in \mcal D} P') = 0$. Thus, we have
	\[
		0 = \phi \left(\sum_{P \in \mcal D} P' \right) = q^2 \phi \left( \sum_{\ell = 1}^{q/2} P_{j, \ell} ' \right).
	\]
By the theory of abelian varieties \cite[Lemma 1]{RS94}, the divisor class group of degree zero of $Y_2$ is $\operatorname{Pic}^0 (Y_2) \cong (\Z / (q+1)\Z) ^{2 g(Y_2)}$.  It follows that
	\[
		(q+1) \phi\left( \sum_{\ell = 1}^{q/2} P_{j, \ell} ' \right) = 0.
	\]
Hence, this lemma will be proved if Equation \eqref{5.6} holds true.

 Since $H (P_\infty) = \langle q/2, q+1 \rangle$, this lemma can be reduced to prove for $m = sq/2 + t (q+1)$, where $0 \leq s \leq q^2-q-2$ and $0 \leq t \leq q/2 -1$. For each $e \in  \F_{q^2}$, let
	\[
	U_e := \{ \alpha\in  \F_{q^2} \colon \alpha^{q+1} =  e \}.
	\]
{\bf Case 1:} If $s \geq 1$, then we can select $c \in \F_{q}^* \setminus \{x_j^{q+1}, x^{q+1}_{j'}\}$ for $1\leq j < j' \leq q^2$, and choose a subset $ J \subseteq T_c=\{\beta\in \F_{q^2}: h(\beta)=c\}$ with $|J| =  t$ and another subset $K \subseteq  \F_{q^2} \setminus ( \{x_j, x_{j'}\} \cup U_c) $ of size $s - 1$. This is possible, since $|T_c|=q/2>t$ and $ q^2 - 2 - q-1 \geq s - 1$. By the proof of Proposition \ref{prop:semigroup}, the function
		\[
		(x - x_j) \prod_{a \in K} (x - a) \prod_{b \in J} (y - b) \in \mcal  L(mP_\infty)
		\]
		has exactly $sq/2 + t(q+1) = m$ distinct zeros in $ {\mcal D}$ from the choices of $K$ and $J$. By Lemma \ref{lem:crit_princ}, we conclude that
		\[
		\phi \left( \sum_{\ell = 1}^{q/2} P_{j, \ell} ' + \sum _{a \in K} \sum _{P \in Z(x - a)} P' + \sum _{b \in J} \sum _{P \in Z(y -b)} P' \right) = 0.
		\]
		In the same way, we can get that
		\[
		\phi \left( \sum_{\ell = 1}^{q/2} P_{j', \ell} ' + \sum _{a \in K} \sum _{P \in Z(x - a)} P' + \sum _{b \in J} \sum _{P \in Z(y -b)} P' \right) = 0.
		\]
		Take the difference of the above two equations and get
		\[
		\phi \left( \sum_{\ell = 1}^{q/2} P_{j, \ell} ' \right) = \phi \left( \sum_{\ell = 1}^{q/2} P_{j', \ell} ' \right).
		\]  	
{\bf Case 2:} If $s = 0$, then $m = t(q + 1)$ for $1 \leq t \leq q/2 - 1$. By Proposition \ref{prop:5.5}, we have
		\[
		D := \sum _{P \in \mcal D} P =  (x - \b^q)_0 + \sum_{\gamma \in \mcal S \setminus T_{\b^{q+1}} } (z_\b - \gamma)_0.
		\]
		After a permutation of $\mcal D$, the above equation still holds true, i.e.,
		\[
		D = \sum _{P \in Z(x - \b^q)} P' + \sum _{\gamma \in \mcal S\setminus T_{\b^{q+1}}} \sum _{P \in Z(z_\b - \gamma)} P'.
		\]
		Since $\phi (D) = 0$, we have
		\[
		\phi \left( \sum _{P \in Z(x - \b^q)} P' \right) = -\phi \left(\sum _{\gamma \in \mcal S \setminus T_{\b^{q+1}} } \sum _{P \in Z(z_\b - \gamma)} P'\right).
		\]
		The lemma can be proved if the right-hand side of the above equation is a constant for all $\b \in \F_{q^2}$.
It will be sufficient to prove
		\begin{equation}
			\phi \left(\sum  _{P \in Z(z_{\b_1} - \gamma_1)} P'\right) = \phi \left( \sum  _{P \in Z(z_{\b_2} - \gamma_2)} P'\right) \label{pf:target_eqn}
		\end{equation}
for any $\b_1, \b_2 \in \F_{q^2}$ and $\gamma_1 \in \mcal S\setminus T_{\b_1^{q+1}}$, $\gamma_2 \in \mcal S \setminus T_{\b_2^{q+1}}$.
		We proceed by subcases.
		\begin{itemize}
			\item[(i)] If $\b_1 = \b_2 = \b \in \F_{q^2}$, we choose a subset $V \subseteq \mcal S \setminus (\{\gamma_1, \gamma_2\} \cup T_{\b^{q+1}})$ of size $t-1$. This is feasible, since $ q^2 / 2 - 2 -q/2 \geq t-1$. By Proposition \ref{prop:5.5} (2),
			\[
			(z_\b - \gamma_1) \prod_{\gamma\in V}  (z_\b -\gamma) \in \mcal L(mP_\infty)
			\] has exactly $m = t(q+1)$ distinct zeros in ${\mcal D}$. By Lemma \ref{lem:crit_princ}, we have
			\[
			\phi \left(\sum _{P\in Z(z_\b - \gamma_1)} P' + \sum _{\gamma \in V} \sum _{P \in Z(z_\b - \gamma)} P'  \right) = 0.
			\]
			Similarly, we have the following equation for $\gamma_2$, that is,
			\[
			\phi \left(\sum _{P\in Z(z_\b - \gamma_2)} P' + \sum _{\gamma \in V} \sum _{P \in Z(z_\b - \gamma)} P'  \right) = 0.
			\]
			Thus, Equation \eqref{pf:target_eqn} holds true by comparing the above two equations.

			\item[(ii)] If $\b_1 \neq \b_2$, then we choose a subset $V \subseteq \mcal S \setminus (\{\gamma_1\} \cup T_{\b_1^{q+1}})$ of size $t-1$ and get
			\begin{equation}\label{eq:3}
			\phi \left(\sum _{P\in Z(z_{\b_1} - \gamma_1)} P' + \sum _{\gamma \in V} \sum _{P \in Z(z_{\b_1} - \gamma)} P'  \right) = 0.
			\end{equation}
			Since the size of $\mcal S \setminus T_{\b_2^{q+1}}$ is
			\[
			|\mcal S \setminus T_{\b_2^{q+1}}| = \frac {q^2 - q}2 > (t - 1)(q+1) = \left|  \bigcup _{\gamma \in A} Z(z_{\b_1} - \gamma ) \right|,
			\]
			we can choose an element $\tilde \gamma \in\mcal S \setminus T_{\b_2}^{q+1}$ such that
			\[
			Z(z_{\b_2} - \tilde  \gamma) \cap \bigcup _{\gamma \in V} Z(z_{\b_1} - \gamma) = \varnothing.
			\]
			By Proposition \ref{prop:5.5}, we obtain $(z_{\b_2} - \tilde  \gamma)\prod_{\gamma\in V}(z_{\b_1} - \gamma)\in \mathcal{L}(mP_\infty)$ and
			\begin{equation}\label{eq:4}
			\phi \left(\sum _{P\in Z(z_{\b_2} - \tilde \gamma)} P' + \sum _{\gamma \in V} \sum _{P \in Z(z_{\b_1} - \gamma)} P'  \right) = 0.				
			\end{equation}
			By comparing Equation \eqref{eq:3} and Equation \eqref{eq:4},
			\[
			\phi \left(\sum _{P\in Z(z_{\b_2} - \tilde \gamma)} P'\right) = \phi \left( \sum _{P\in Z(z_{\b_1} - \gamma_1)} P' \right).
			\]
			From the subcase (i), we have
			\[
			\phi \left(\sum _{P\in Z(z_{\b_2} - \tilde \gamma)} P' \right) = \phi \left( \sum _{P\in Z(z_{\b_2} - \gamma_2)} P' \right). 	
			\]
			Thus, Equation \eqref{pf:target_eqn} holds true, and the proof of Case $2$ is complete.
		\end{itemize}

\end{proof}


The following lemma is useful to determine the automorphism of rational algebraic geometry codes given in \cite[Lemma 3.5]{Sti90}.
\begin{lemma}
	If $C(\mcal D, E) = C(\mcal D', E')$ are rational algebraic geometry codes over $\F_{q^2}$ of length $n$, and $E,E'$ are divisors of $\F_{q^2}(x)$ satisfying that  $ 1\leq  \deg E= \deg E' \leq n - 3$, then there exists an automorphism $ \sigma \in \mrm {PGL}_2 (\F_{q^2})$ such that $\sigma (\mcal D) = \mcal D'$. \label{lem:rat_codes}
\end{lemma}

The following lemma is useful to determine the automorphism of algebraic geometry codes from the Abd\'on--Torres function field.
\begin{lemma}
	\begin{enumerate}
		\item Let $a_1, a_2 \in {\F_{q^2}}$, and $p_1(x), p_2(x) \in {\F_{q^2}}[x]$ be polynomials of degree less than $3$. If there exists a subset $J \subseteq \{1,2,\dots, q^2\}$ of size at least $4$ such that for all $j \in J$ and $\ell \in \{1,2,\dots, q/2\}$,
		\[
		a_1 y_{j, \ell } + p_1(x_j) = a_2 y_{j, \ell} + p_2(x_j),
		\]
		then $a_1 = a_2$ and $p_1(x) = p_2(x)$. \label{item:unique}
		\item Let $a, c\in \F_{q^2}^*$, $b \in {\F_{q^2}}$, and $s(x)\in {\F_{q^2}}[x]$ be a polynomial with degree $\le 3$. If
		\[
		h(c y_{j, \ell} + s(x_j)) = (a x_j + b)^{q+1}
		\]
		for each rational place $P_{x_j, y_{j,\ell}}$ of $Y_2$, then $c = 1$ and
		\[
		s(x) = a^2 b^{2q}x^2 + a b^qx + e
		\] for some $e \in {\F_{q^2}}$ satisfying $b^{q+1} = h(e) $. \label{item:conform}
	\end{enumerate} \label{lem:comparison}
\end{lemma}

\begin{proof}
		\begin{enumerate}
		\item By Proposition \ref{prop:semigroup} and the given conditions, the function
		\[
		a_1 y + p_1(x) - a_2 y - p_2(x) \in \mcal L\left(\frac{3q}{2} P_\infty\right)
		\] has at least $2q$ zeros but only one pole of order at most $3q/2$. Thus, it is the zero function from \cite[Theorem 1.4.11]{Sti2009}.
This result follows from the fact that $y, 1,x, x^2$ and $ x^3$ are linearly independent over $\F_{q^2}$.
		\item By Proposition \ref{prop:semigroup} and the given conditions, the function
		\[
		h(cy + s(x)) - (ax + b)^{q+1} \in \mcal L(q^2 P_\infty)
		\] has at least $q^3 / 2$ zeros but at most $q^2$ poles. Thus, it is the zero function. Write $s(x) = dx^3 + Q (x)$, where $Q (x)$ is a degree $2$ polynomial over $\F_{q^2}$ and $d \in \F_{q^2}$.  By considering their discrete valuations, we compare coefficients of monomials $$1,x, x^2, \dots, x^{3q/2}, y, y^2, \dots, y^{q/2}.$$ Note that $\nu_{P_\infty} (x^{3q / 2}) < \nu_{P_\infty} (x^{q+1})=\nu_{P_\infty} (y^{q/2})$, then $d = 0$. Now this lemma follows from the proof of Proposition \ref{prop:struct_stab}.
	\end{enumerate}
\end{proof}

\begin{theorem}\label{thm:5.8}
	Let $q = 2^n \geq 4$ and $Y_2 / \F_{q^2}$ be the maximal Abd\'on--Torres function field $\F_{q^2}(x,y)$ defined by
	\[
		x^{q+1} = y^{q/2} + y^{q/4} + \dots + y^2 + y.
	\]
	Let $\mcal D = \{P_{j, \ell}: 1\le j\le q^2, 1\le \ell\le q/2\}$ be an ordered set of all finite rational places of $Y_2$. For any integer
$ q+1 \leq m \leq (q^3 - 3q - 2) / 2$,
let\[C_m = C(\mcal D, mP_\infty)\]
	be the one-point algebraic geometry code given by the field $Y_2$. Then the automorphism group of $C_m$ is isomorphic to the automorphism group $\Aut (Y_2 / \F_{q^2})$ of $Y_2$ over $\F_{q^2}$.
\end{theorem}

\begin{proof}
Let $\tau $ be an automorphism of $C_m$. Our goal is to show that the automorphism $\tau$ induces an $\F_{q^2}$-automorphism $\sigma$ of $Y_2$ determined by $\sigma(x)=ax+b$ and $\sigma(y)=y + (ab^q x )^2 + ab^q x + c$ given in Proposition \ref{prop:stabilizer} and Theorem \ref{thm:mainresult}.

Let $\mcal D'= \{ P_{j, \ell}' : 1\le j\le q^2, 1\le \ell\le q/2\}$ be the resulting list after the permutation induced by $\tau$, i.e., $P_{i,j}^\prime=\tau(P_{i,j})$. Then $ C (\mcal D', mP_\infty) = C  (\mcal D, mP_\infty)$.  By Lemma \ref{lem:5.6}, there exists a function $w_j \in \mcal L( q/ 2 \cdot  P_\infty)$ for each $1\le j \le q^2$ such that
\[(w_j)  = -\frac q 2 P_\infty + \sum _{\ell = 1}^{q/2} P_{j, \ell}'.\]
Since $\{1,x\}$ is a basis of $\mcal L( q /2 \cdot P_\infty)$, we can choose $w_j=x-x_j'$ with $x_j' = x(P_{j, \ell}')\in \F_{q^2}$ for all $1\le \ell \le q/2$.
Let the $y$-coordinate of $P_{j,\ell}'$ be $y_{i,j}'$, i.e., $P_{j,\ell}'=P_{x_j',y_{j,\ell}'}$.
In order to obtain $\sigma(x)$ induced by the automorphism $\tau$ by Lemma \ref{lem:rat_codes}, we need to design two rational algebraic geometry codes.
For each polynomial $f \in \F_{q^2}[T]$ of degree $\leq \gauss {2m / q}$, there exists a function $g(x,y) \in  \mcal L(m P_\infty)$ such that
$$ g(x,y)(P_{j, \ell}')=f(x)(P_{j, \ell})=f(x_j)$$ for $1 \leq j \leq q^2$ and $1 \leq \ell \leq q/2$.
This is possible, since $ f(x) \in  \mcal L(mP_\infty)$ and $C  (\mcal D', mP_\infty) = C (\mcal D , mP_\infty)$. By Lagrange's interpolation formula, one can find a polynomial $\tilde f(x) \in \F_{q^2}[T]$ of degree $\leq \gauss {2m / q}$ such that
\[\tilde f (x_j')=f(x_j)  \text{ for all } 1\leq j \leq \gauss{ \frac {2m} q} + 1.\]
Since $g(x,y)(P_{j, \ell}')=f(x_j)=\tilde f (x_j')=\tilde f (x)(P_{j, \ell}')$, it is clear that
\[g(x,y) - \tilde f(x) \in \mcal  L\left(
m P_\infty - \sum _{j = 1}^{\gauss {2m/q} + 1} \sum_{\ell = 1}^{q/2} P_{j, \ell} '\right)=\{0\}.\]
Thus, $g(x,y) = \tilde f(x)$ and hence $\tilde f(x_j')=f(x_j)$ for all $1\le j\le q^2$.
Let $\mcal R=\{x_1,x_2,\cdots,x_{q^2}\}$ and $\mcal R'=\{x_1',x_2',\cdots,x_{q^2}'\}$ be two ordered sets.
Hence, there are two rational algebraic geometry codes which are equal:
\[C  \left(\mcal R, \gauss {\frac {2m} q} (x)_{\infty}^{\F_{q^2}(x)} \right) = C  \left(\mcal R', \gauss {\frac {2m} q} (x)_{\infty}^{ \F_{q^2}(x) }\right).\]
By Lemma \ref{lem:rat_codes}, there exists an automorphism $\sigma\in \mrm {PGL}_2 (\F_{q^2})$ such that $\sigma (\mcal R) = \mcal R'$.
Moreover, $\mcal R=\mcal R'=\F_{q^2}$ as sets, the automorphism $\sigma$ must be an affine linear transformation determined by $\sigma(x)=ax+b$, i.e.,
there are $a \in \F_{q^2}^*$ and $b \in \F_{q^2}$ such that
\[
 x_j' = a x_j + b \text{ for all } 1 \leq j \leq  q^2.
\]

In order to lift this automorphism of $\F_{q^2}(x)$ to an automorphism of $Y_2$, let us consider $$ z = y \prod _{i \in I} (x - x_i)\in  \mcal L(mP_\infty)$$ for any index set $ I \subseteq \{1,2,\dots, q^2\}$ with $ |I| = \gauss {(m - q - 1)/ (q/2)}>0$. Since 	$C  (\mcal D', mP_\infty) = C (\mcal D, mP_\infty)$, one can find a function $w \in  \mcal L(mP_\infty)$ such that
\begin{equation}\label{eq:automorphism}
 w (P_{j, \ell} ' )=	z (P_{j,\ell}) .
\end{equation}
for all $ 1 \leq j \leq q^2$ and $ 1 \leq \ell \leq q/2$.
Since $ h(y)=x^{q+1}$, it can be written as $$ w = \sum _{k = 0}^{q/2-1} w_{k} (x) y^k, $$ where $ w_k (x) \in \F_{q^2}[x]$.
For any $j\in I$ and all $1 \leq \ell \leq q/2$, we obtain
\begin{equation}
0=	y_{j, \ell} \prod _{i \in I} (x_j - x_i) = \sum_{k = 0} ^{q/2-1} w_k (x_j') (y_{j, \ell}' )^k. \label{eqn:eval}
\end{equation}
from Equation \eqref{eq:automorphism}.
It follows that the polynomial $\sum_{k = 0} ^{q/2-1} w_k (x_j')y^k$ is the zero polynomial, i.e., $w_k (x_j')=0$ for all $j\in I$ and $1\le k\le q/2$.
By considering the discrete valuation of $w$ at $P_\infty$, we have
$$\deg_x w_k(x) \leq \gauss { \dfrac {m - k (q+1)} {q/2} } \text{ or } w_k(x)=0.$$
By Lagrange's interpolation formula, it is easy to see that $w_k(x) = 0$ for $k \geq 2$.
From the estimation of the degrees of $w_k(x)$, there exists an element $ a_I \in \F_{q^2}^*$ such that $$w_1 (x) = a_I \prod _{i\in I} (x - x_i ')$$ and a polynomial $ p_I(x) \in \F_{q^2}[x]$ of degree $\leq 3$ such that $$w_0 (x) = p_I(x) \prod_{i\in I}(x - x_i').$$
For $ i \notin I$ and $1\le \ell\le q/2$, one can obtain
$$y_{j, \ell} \prod _{i \in I} (x_j - x_i) =(a_I\cdot y_{j,\ell}'+p_I(x_j')) \prod _{i \in I} (x_j' - x_i') $$
from Equation \eqref{eq:automorphism}.
Since $x_j' = ax_j + b$ for all $1\le j\le q^2$, one has
\begin{equation}\label{eq:5-8-7}
y_{j, \ell} = a^{ |I|} (a_I y_{j,\ell }'+ p_I(x_j')).
\end{equation}
for $ j \notin I$ and $1\le \ell \le q/2$.

Now we claim that Equation \eqref{eq:5-8-7} holds true for any other index set $J$ with $| J |= |I|$ and $  |I \cap J |= | I| - 1$, whenever $ | I| \geq 1$.
By the inclusion-exclusion principle, the cardinality of the set $\{1, 2,\dots, q^2\} \setminus (I \cup J)$ is at least
\begin{align*}
	| \{1,2, \dots, q^2\}\setminus (I \cup J)| &= q^2 - | I| - | J| + |I \cap J|\\
	&= q^2 - \gauss {\frac {m - (q+1)} {q/2}} - 1 \\
	&\geq q^2 - \gauss {\frac {q^3 + 2q^2 - 8q - 8}{2q}} - 1\\
	&\geq q^2 - \left(\frac {q^2}2 + q - 4\right) - 1 \\
	&= \frac {q^2}2 - q + 4 - 1 \geq 7.
\end{align*}
For each $j \in \{1, 2, \dots, q^2\} \setminus (I \cup J)$ and $1\le \ell \le q/2$, we have
$$\begin{cases}
	y_{j, \ell} = a^{|I|}(a_I y_{j,\ell }'+ p_I(x_j') ) , \\
	y_{j, \ell} = a^{|J|}(a_J y_{j,\ell }'+ p_J(x_j') ) .
\end{cases}$$
By Lemma \ref{lem:comparison} \ref{item:unique}, we conclude that $a_I = a_J$  and $p_I = p_J$.
By taking all such index subsets, Equation \eqref{eq:5-8-7} holds for all $1\le j\le q^2$ and $1\le \ell\le q/2$.

Consider the following substitution of variables:
\[
\begin{cases}
	x' = ax + b, \\
	y' = a^{-1}_I ( a^{-| I|}y + p_I (ax + b)).
\end{cases}
\]
By evaluating on the rational place $ P_{j,\ell}'$, it is clear that $x (P_{j, \ell}') = x_j'= ax_j + b$ and $y (P_{j, \ell}') = y' _{j, \ell}$.
Since $h(y)=x^{q+1}$, we obtain $h(y_{j, \ell}')=(x_j')^{q+1}$ for $1\le j\le q^2$ and $ 1\le \ell\le q/2$, i.e.,
\[h (a_I^{-1} (a^{-|I|} y_j + p_I(ax_j +b)  )) = (ax _j + b)^{q+1}.\]
By Lemma \ref{lem:comparison} \ref{item:conform}, we know that
\[
a^{-1}_I a^{- |I|} = 1, \quad a^{-1}_I p_I (ax +b) = (ab^q x )^2 + ab^q x + c
\]
where $ c \in \F_{q^2}$ and $h(c) = b^{q+1}$. Thus, the automorphism $\tau $ of $C_m$ induces an $\F_{q^2}$-automorphism $\sigma$ of $Y_2$ determined by $\sigma(x)=ax+b$ and $\sigma(y)=y + (ab^q x )^2 + ab^q x + c$.

Let $D = \sum_{P \in\mcal D} P$. For any automorphism $\sigma\in \Aut (Y_2 / \F_{q^2})$, we have
$\sigma(D)=D$ and $\sigma(mP_\infty)=mP_\infty$, i.e., $\Aut_{D,mP_\infty} (Y_2 / \F_{q^2})=\Aut (Y_2 / \F_{q^2})$.
By Proposition \cite[Proposition 8.2.3]{Sti2009}, any automorphism in $\Aut (Y_2 / \F_{q^2})$ induces an automorphism of the code $C_m$.
Hence, we have proved that $\Aut (C_m) \cong \Aut (Y_2 / \F_{q^2})$ for $ q+1 \leq m \leq (q^3 - 3q - 2) / 2$.
\end{proof}

Combining Theorem \ref{prop:aut_code_easy} with Theorem \ref{thm:5.8}, we determine the automorphism group $\Aut (C_m)$ completely for all positive integers $m$. Thus, we have shown that $\Aut (C_m) \cong \Aut (Y_2 / \F_{q^2})$ for $ q+1 \leq m \leq (q^3 +q^2- 4q - 2) / 2$.

\subsection*{Disclaimer} The creation of this article makes NO use of any content-generated AI.

\end{document}